\documentclass[reqno]{amsart}
\usepackage{amsmath,amsfonts,amssymb,amsthm,amscd,yfonts, mathtools}
\usepackage{mathtools}
\usepackage{tensor}
\usepackage[makeroom]{cancel}
\usepackage[margin=1.6in]{geometry}
\usepackage{enumitem}
\usepackage{tabu}
\usepackage{tikz, tikz-cd}

\definecolor{darkblue}{HTML}{111199}
\definecolor{darkgreen}{HTML}{336633}
\definecolor{darkred}{HTML}{993333}
\definecolor{darkpurple}{HTML}{995599}

\makeatletter

\def\inc{\operatorname{in}}
\def\pr{\operatorname{pr}}
\def\half{{\textstyle\frac{1}{2}}}

\def\isig{i\!\operatorname{-sig}}
\def\la{\lambda}
\def\eps{\varepsilon}
\def\0{{\bar{0}}}
\def\1{{\bar{1}}}
\def\ba{\text{\boldmath$a$}}
\def\bb{\text{\boldmath$b$}}
\def\bbeta{\text{\boldmath$\beta$}}
\def\bc{\text{\boldmath$c$}}
\def\bi{\text{\boldmath$i$}}
\def\bd{\text{\boldmath$d$}}
\def\bj{\text{\boldmath$j$}}
\def\bsigma{\text{\boldmath$\sigma$}}

\def\g{{\mathfrak{g}}}
\def\b{{\mathfrak{b}}}
\def\t{{\mathfrak{t}}}
\def\h{{\mathfrak{h}}}
\def\C{{\mathcal{C}}}
\def\sO{s\mathcal{O}}

\def\O{\mathcal{O}}
\def\ttQ{{\tt Q}}
\def\ttM{{\tt M}}
\def\ob{\operatorname{ob}}
\def\CT{{\mathtt{CT}}}
\def\u{V}

\def\pr{{\operatorname{pr}}}
\def\sF{s\hspace{-0.26mm}F}
\def\sE{s\hspace{-0.26mm}E}

\def\AHC{{\mathcal{AHC}}}
\def\QHC{{\mathcal{QHC}}}
\def\QH{{\mathcal{QH}}}

\newtheorem{theorem}{Theorem}[section]
\newtheorem{lemma}[theorem]{Lemma}

\newtheorem{corollary}[theorem]{Corollary} 
  
\theoremstyle{definition}  
\newtheorem{definition}[theorem]{Definition}

\newtheorem{example}[theorem]{Example}

\newtheorem{remark}[theorem]{Remark}

\@addtoreset{equation}{section}

\newcommand{\arxiv}[1]{{\tt{arXiv:#1}}}

\DeclareSymbolFont{bbold}{U}{bbold}{m}{n}
\DeclareSymbolFontAlphabet{\mathbbold}{bbold}
\DeclareMathOperator\End{End}
\DeclareMathOperator\Hom{Hom}
\DeclareMathOperator\Ext{Ext}

\DeclareMathOperator\N{\mathbb{N}}
\DeclareMathOperator\Z{\mathbb{Z}}
\DeclareMathOperator\CC{\mathbb{C}}
\DeclareMathOperator\K{{\Bbbk}}
\DeclareMathOperator\id{{\operatorname{id}}}

\DeclareMathOperator\D{\mathcal{D}}
\DeclareMathOperator\B{\mathbf{B}}

\DeclareMathOperator\wt{wt}
\DeclareMathOperator\WT{\mathbf{wt}}

\DeclareMathOperator\Q{\mathbb{Q}}

\DeclareMathOperator\lmod{\text{-}mod}
\DeclareMathOperator\lsmod{\text{-}smod}

\def\q{{\mathfrak{q}}}
\DeclareMathOperator\ssl{\mathfrak{sl}}

\begin{document}

\title[Type A blocks]{\boldmath Type A blocks of super category $\O$}
\author{Jonathan Brundan and Nicholas Davidson}

\begin{abstract}
We show that every block of category $\O$ for the 
general linear Lie superalgebra $\mathfrak{gl}_{m|n}(\K)$
is equivalent to some corresponding block
of category $\O$ for the queer Lie superalgebra $\q_{m+n}(\K)$.
This implies the truth of the Kazhdan-Lusztig conjecture 
for the so-called type A blocks of category $\O$ for the queer Lie superalgebra
as formulated by Cheng, Kwon and Wang.
\end{abstract}

\thanks{Research supported in part by NSF grant DMS-1161094.}

\maketitle

\section{Introduction}

In this article, we study the analog of the BGG category $\O$ for the Lie superalgebra $\mathfrak{q}_n(\K)$.
Recent work of Chen \cite{C} has reduced most questions about this
category just to the
study of three particular types of block, which we refer to here as
the type A, type B and
type C blocks.
Type B blocks (which correspond to integral weights) 
were investigated already by the first author in
\cite{B1}, leading to a Kazhdan-Lusztig conjecture for characters
of irreducibles in such blocks in terms of certain canonical bases for
the quantum group of type B$_\infty$.
In \cite{CKW}, Cheng, Kwon and Wang formulated 
analogous conjectures for the type A blocks (defined below) 
and the
type C blocks (which correspond to half-integral weights)
in terms of canonical bases of quantum groups of types A$_\infty$ and C$_\infty$,
respectively. 

The main goal of the article is to prove the
Cheng-Kwon-Wang conjecture 
for type A blocks (\cite[Conjecture 5.14]{CKW}).
To do this, we use some tools from higher representation theory
to establish an equivalence of categories
between the type A blocks of category $\mathcal O$ for the Lie
superalgebra $\mathfrak{q}_n(\K)$ and 
integral blocks of category
$\mathcal O$ for a general linear Lie superalgebra. 
This reduces the Cheng-Kwon-Wang
conjecture for type A blocks to 
the Kazhdan-Lusztig conjecture of \cite{B0}, which was 
proved already in \cite{CLW, BLW}. 

Regarding the types B and C conjectures,
Tsuchioka discovered in 2010 that the type B
canonical bases considered in \cite{B1} fail to satisfy appropriate
positivity properties, so that 
the conjecture from \cite{B1} is certainly false. 
Moreover, 
after the first version of \cite{CKW} appeared, Tsuchioka
pointed out similar issues with the type C
canonical bases studied in \cite{CKW}, so that the
Cheng-Kwon-Wang conjecture for type C blocks as
formulated in the first version of their article (\cite[Conjecture 5.10]{CKW}) 
also seems likely to be incorrect.

In fact, the techniques developed in this article
can be applied also to the study of the type C blocks.
This will be spelled out
in a sequel to this paper \cite{BDC}.
In this sequel, we prove a modified version of the Cheng-Kwon-Wang conjecture
for type C blocks: one needs to replace Lusztig's
canonical basis with Webster's ``orthodox basis''
arising from the indecomposable
projective modules of the
tensor product algebras of \cite[$\S$4]{Web}.
This modified conjecture was proposed independently by Cheng,
Kwon and Wang in a revision of their article (\cite[Conjecture
5.12]{CKW}). 
It is not 
as satisfactory as the situation for type A blocks, since there is no
elementary algorithm to compute Webster's basis
explicitly (unlike the canonical basis).
Also in the sequel,
we will prove \cite[Conjecture 5.13]{CKW}, and settle
\cite[Question 5.1]{CKW} by identifying the
category of finite-dimensional half-integer weight representations of
$\q_n(\K)$ 
with a previously known highest weight category (as 
suggested by \cite[Remark 6.7]{CK}).

There is more to be said about type B blocks too; in fact,
these are the most intriguing of all. Whereas the types A and
C blocks carry the additional structure of tensor product
categorifications
in the sense of \cite{LW, BLW}
for the infinite rank Kac-Moody algebras of 
types A$_\infty$ and C$_\infty$, respectively,
the type B blocks produce an example of a tensor product categorification of
{\em odd} type B$_\infty$,
i.e. one needs
a {\em super} Kac-Moody 2-category in the sense of \cite{BE2}.
This will be developed in subsequent work by the second author.

\vspace{2mm}

In the remainder of the introduction, we are going to formulate our main
result for type A blocks in more detail.
To do this, we first briefly recall some basic notions of superalgebra.
Let $\K$ be a ground field which is algebraically
closed of characteristic zero, and fix a choice 
of $\sqrt{-1}\in\K$.
We adopt the language of \cite[Definition 1.1]{BE}:
\begin{itemize}
\item
A
{\em supercategory} is a category enriched in the symmetric
monoidal category of vector superspaces, i.e. the category of $\Z/2$-graded
vector spaces over $\K$ with morphisms that are parity-preserving
linear maps.
\item
Any morphism in a supercategory decomposes uniquely into an even
and an odd morphism as $f = f_\0 + f_\1$. 
 A {\em superfunctor} between
supercategories means a $\K$-linear functor which preserves the
parities of morphisms.  
\item
For superfunctors
$F, G : \C \to \D$,
a \textit{supernatural transformation} $\eta : 
F \Rightarrow G$ is a family of morphisms 
$\eta_M = \eta_{M,\0} + \eta_{M,\1} : FM \to GM$ for each $M\in\ob \C$, such that
$\eta_{N,p}  \circ Ff = (-1)^{|f|p} Gf \circ \eta_{M,p}$
for every homogeneous morphism $f:M \rightarrow N$
in $\C$ and each $p \in \Z/2$.
\end{itemize}
For any supercategory $\C$, the {\em Clifford twist}
$\C^{\CT}$ is the supercategory whose objects are pairs 
$(X, \phi)$ for $X \in \ob \C$
and $\phi \in \End_\C(X)_\1$ with $\phi^2 = \id$, and whose morphisms
$f:(X,\phi) \rightarrow (X',\phi')$ are morphisms $f:X \rightarrow X'$ in
$\C$ 
such that $f_p \circ \phi = (-1)^p \phi'\circ f_p$
for each $p \in \Z/2$.
One can also take Clifford twists of superfunctors and supernatural
transformations (details omitted),
so that $\CT$ is actually a 2-superfunctor from the 2-supercategory of
supercategories to itself in the sense of \cite[Definition 2.2]{BE}.
The following basic lemma is a variation on \cite[Lemma 2.3]{KKT}.

\vspace{2mm}
\noindent
{\bf Lemma.} 
{\em
Suppose $\C$ is a supercategory such that
\begin{itemize}
\item
$\C$ is additive;
\item $\C$ is $\Pi$-complete, i.e. every object of $\C$ is the target
  of an odd isomorphism;
\item all even idempotents split.
\end{itemize}
Then the supercategories $\C$ and $(\C^\CT)^\CT$ are superequivalent.
}
\vspace{2mm}

For example, suppose that $A$ is a locally unital superalgebra, i.e. an
associative superalgebra $A = A_\0 \oplus A_\1$ equipped with a
distinguished collection $\{1_x\:|\:x \in X\}$ of mutually orthogonal
even idempotents
such that $A = \bigoplus_{x, y \in X} 1_y A 1_x$.
Then there is a supercategory $A\lsmod$
consisting of finite-dimensional left $A$-supermodules $M$ which are 
locally unital in the sense that $M = \bigoplus_{x \in X} 1_x M$.
Even morphisms in $A\lsmod$ are parity-preserving 
linear maps such that $f(av) = a f(v)$ for all $a \in A, v \in M$; odd morphisms are parity-reversing linear maps
such that $f(av) = (-1)^{|a|} a f(v)$ for homogeneous $a$.
There is an obvious isomorphism between the Clifford twist 
$A\lsmod^\CT$  of this supercategory
and the supercategory $A\otimes C_1 \lsmod$,
where $C_1$ denotes the rank one Clifford superalgebra generated by an
odd involution $c$, and $A \otimes C_1$ is the usual braided tensor
product of superalgebras.
Hence, $(A\lsmod^\CT)^\CT$ is isomorphic to $A \otimes C_2 \lsmod$
where $C_2 := C_1 \otimes C_1$ is the rank two Clifford
superalgebra generated by 
$c_1 := c \otimes 1$ and $c_2
:= 1 \otimes c$.
In this situation, the above lemma is obvious as $A \otimes C_2$ is
isomorphic to the matrix superalgebra $M_{1|1}(A)$, which is Morita
superequivalent to $A$.

Now fix $n \geq 1$ and let
$\mathfrak{g} = \g_\0\oplus\g_\1$ be the Lie superalgebra $\q_n(\K)$.
Recall this is the subalgebra of the general linear Lie superalgebra
$\mathfrak{gl}_{n|n}(\K)$ consisting of all matrices of block form
\begin{equation}\label{blockform}
\left( \begin{array}{c|c} A & B \\\hline B & A \end{array} \right).
\end{equation}
Let $\mathfrak{b}$ (resp. $\mathfrak{h}$) be the standard Borel (resp.
Cartan) subalgebra of $\mathfrak{g}$ consisting of all matrices (\ref{blockform})
in which $A$ and $B$ are upper triangular (resp. diagonal).
Let $\mathfrak{t} := \mathfrak{h}_{\0}$.
We let $\delta_1,\dots,\delta_n$ be the basis for $\mathfrak{t}^*$
such that $\delta_i$ picks out the $i$th diagonal entry of the matrix
$A$. Fix also a {sign sequence}
$\bsigma = (\sigma_1,\dots,\sigma_n)$ with each $\sigma_r \in \{\pm \}$,
and a scalar $z \in \K$ such that $2z \notin \Z$.
{\bf We stress that all of our subsequent notation depends implicitly on these choices.}

It will be convenient to index certain weights in $\mathfrak{t}^*$
by the set $\mathbf{B} := \Z^n$ via the following 
{\em weight dictionary}:
for $\bb = (b_1,\dots,b_n) \in \B$
let
\begin{equation}\label{wtdict1}
\lambda_\bb := \sum_{r=1}^n \lambda_{\bb,r} \delta_r
\quad\text{where}\quad \lambda_{\bb,r} := \sigma_r(z+b_r).
\end{equation}
We let $\sO$
be the category of all $\mathfrak{g}$-supermodules $M$ satisfying the
following properties:
\begin{itemize}
\item
$M$ is finitely generated as a $\mathfrak{g}$-supermodule;
\item
$M$ is locally finite-dimensional over $\mathfrak{b}$;
\item
$M$ is semisimple over $\t$ with all weights 
of the form $\lambda_\bb$ for $\bb \in \B$.
\end{itemize}
Morphisms in $\sO$ 
are arbitrary 
(not necessarily even) $\mathfrak{g}$-supermodule homomorphisms, so
that it is a supercategory.
It also admits a parity switching functor $\Pi$.
The {\em type A blocks} mentioned earlier are the blocks of 
$\sO$ for all possible choices of $\bsigma$ and $z$.

For each $\bb \in \B$, there is an
irreducible supermodule $L(\bb) \in \ob \sO$
of highest weight $\lambda_\bb$. 
Note the highest weight space of
$L(\bb)$ is not 
one-dimensional: it is some sort of irreducible Clifford supermodule
over the Cartan subalgebra $\h$.
Every irreducible supermodule in $\sO$ is isomorphic 
to $L(\bb)$ for a unique $\bb \in \B$
via a homogeneous (but not necessarily even) isomorphism.
If $n$ is odd, $L(\bb)$ is of type
$\ttQ$, i.e., $L(\bb)$ is evenly isomorphic to its parity flip $\Pi L(\bb)$. 
When $n$ is even, the irreducible $L(\bb)$ is of type $\ttM$, and
we should explain how to
distinguish it from its parity flip.
For each $i \in \Z$, we fix a choice $\sqrt{z+i}$ of a square root of $z+i$,
then set $\sqrt{-(z+i)} := (-1)^{i} \sqrt{-1} \sqrt{z+i}$. The key
point about this is that
\begin{equation}\label{violin}
\sqrt{-(z+i)} \sqrt{-(z+i+1)} = \sqrt{z+i}\sqrt{z+i+1}
\end{equation}
for each $i \in \Z$.
Let $d_r' \in \g_\1$ be the matrix of the form (\ref{blockform}) such that $A =
0$ and $B$ is the $r$th diagonal matrix unit.
Then, for even $n$, 
we assume that $L(\bb)$ is chosen so that $d_1' \cdots d_n'$ acts on
any even highest weight vector by the scalar
$(\sqrt{-1})^{n/2} \sqrt{\lambda_{\bb,1}}
\cdots
\sqrt{\lambda_{\bb,n}}$.
This determines $L(\bb)$ uniquely up to even isomorphism.

Turning our attention to the category on the other side of our main
equivalence,
let $\mathfrak{g}'$ be the general linear Lie superalgebra 
consisting
of $n \times n$ matrices under the supercommutator, with
$\Z/2$-grading defined by declaring that the $rs$-matrix unit is
even if $\sigma_r = \sigma_s$ and odd if $\sigma_r \neq
\sigma_s$.
Let $\mathfrak{b}'$ (resp. $\mathfrak{t}'$) be the standard Borel
(resp. Cartan) subalgebra consisting of upper triangular
(resp. diagonal) matrices in $\mathfrak{g}'$.
As before, we let $\delta_1',\dots,\delta_n'$ be the basis for
$(\mathfrak{t}')^*$ defined by the diagonal coordinate functions.
We introduce another {\em weight dictionary}
 (which in this setting is some ``signed $\rho$-shift''): 
for $\bb \in \B$, let
\begin{equation}\label{wtdict}
\lambda_\bb' := 
\sum_{r=1}^n \lambda_{\bb,r}' \delta_r'
\quad\text{where}\quad
\lambda_{\bb,r}' := 
\sigma_r \left(b_r
+ \sigma_1 1 +\cdots+\sigma_{r-1} 1 +\half(\sigma_r 1-1)\right).
\end{equation}
Let $\sO'$ be the supercategory of $\g'$-supermodules $M'$ such that
\begin{itemize}
\item $M'$ is finitely generated as a $\g'$-supermodule;
\item $M'$ is locally finite-dimensional over $\b'$;
\item $M'$ is semisimple over $\t'$ with all weights of the form
  $\lambda'_\bb$
for $\bb \in \B$.
\end{itemize}
Note $\sO'$ is the sum of all of the blocks of the 
usual category $\O$ for $\g'$
corresponding to integral weights of $\t'$.
For each $\bb \in \B$, there is a unique (up to even isomorphism)
irreducible
supermodule $L'(\bb) \in \ob \sO'$ 
generated by a homogeneous 
highest weight vector of weight $\lambda'_\bb$ and
parity 
$
\sum_{\sigma_r = -} \lambda'_{\bb,r} \pmod{2}.
$

\vspace{2mm}
\noindent
{\bf Main Theorem.}
{\em
If $n$ is even then there is a superequivalence
$\mathbb{E}:\sO \rightarrow \sO'$ 
such that 
$\mathbb{E} L(\bb)$ is evenly isomorphic to $L'(\bb)$ for each $\bb
\in \B$.
If $n$ is odd then there is a superequivalence
$\mathbb{E}:\sO \rightarrow (\sO')^\CT$
such that $\mathbb{E} L(\bb)$ is evenly isomorphic to
$(L'(\bb)\oplus \Pi L'(\bb), \phi)$ for either of the two choices of odd
involution $\phi$.
}
\vspace{2mm}

If $\C$ is any $\K$-linear category,
we let $\C \oplus \Pi \C$ be the supercategory whose objects are
formal direct sums
$V_1 \oplus \Pi V_2$ for $V_1, V_2 \in \ob \C$,
with morphisms $V_1 \oplus \Pi V_2 \rightarrow W_1 \oplus \Pi W_2$ 
being matrices 
of the form 
$f = 
\left( \begin{array}{cc} f_{11} & f_{12} \\ f_{21} & f_{22} \end{array} \right)$
for $f_{ij} \in \Hom_\C(V_j,W_i)$.
The $\Z/2$-grading is defined so that $f_\0 = 
\left( \begin{array}{cc} f_{11} & 0 \\ 0 & f_{22} \end{array} \right)$
and $f_\1 = \left( \begin{array}{cc} 0 & f_{12} \\f_{21} & 0 \end{array}
\right)$.
For example, if $\C$ is the category $A\lmod$ of
finite-dimensional locally unital
modules over some locally unital
algebra $A$, then $\C \oplus \Pi \C$ may be identified with the
category $A\lsmod$, viewing $A$ as a purely even superalgebra.

It was noticed originally in \cite{B0} that the category $\sO'$
can be decomposed in this way:
let $\O'$ be full subcategory of $\sO'$
consisting of all $\g'$-supermodules whose $\lambda'_\bb$-weight space
is concentrated in parity 
$\sum_{\sigma_r = -} \lambda'_{\bb,r} \pmod{2}$ for each
$\bb \in \B$; obviously, there are no non-zero odd morphisms between objects of $\O'$.
Then 
$\sO'$ decomposes as $\sO' = \O' \oplus \Pi \O'$.
Moreover, $\O'$ 
is a highest weight category with 
irreducible objects $\{L'(\bb)\:|\:\bb \in \B\}$
indexed by the set $\B$ as above.
In fact, $\O'$ equivalent to $A\lmod$ for a locally
unital algebra $A$ such that the left ideals $A
1_x$ and right ideals $1_x A$ are finite-dimensional
for all distinguished idempotents $1_x \in A$.
Although not needed here,
the results of \cite{BLW} imply further that the algebra $A$ may be
equipped with a $\Z$-grading making it into a (locally unital) {\em Koszul algebra}; this leads to the definition of a graded analog of
the category $\O'$ similar in spirit to Soergel's graded lift of
classical category $\O$ as in e.g. \cite{BGS}.

Combining these remarks with our Main Theorem, we deduce:
\begin{itemize}
\item
For even $n$, the
category $\sO$ decomposes as $\sO = \O \oplus \Pi \O$, where $\O$
is the Serre subcategory generated by the irreducible supermodules
$\{L(\bb)\:|\:\bb \in \B\}$ introduced above (but {\em not} their
parity flips).
Moreover, $\O$ is equivalent to $\O'$, hence, to 
the category $A\lmod$ where $A$ is the
Koszul algebra just introduced.
\item
For odd $n$, $\sO$ is superequivalent
to $A \otimes C_1 \lsmod$, viewing $A$ 
as a purely even superalgebra.
This implies that the underlying category $\underline{\sO}$
consisting of the same objects as $\sO$ but only its even morphisms
is equivalent to $A \lmod$, hence, to $\O'$.
\end{itemize}
As already mentioned, the Kazhdan-Lusztig conjecture for $\sO$
formulated in \cite{CKW} follows immediately from this discussion
together with the Kazhdan-Lusztig conjecture for $\O'$ 
proved in \cite{CLW, BLW}.

There is also a parabolic analog of our Main Theorem.
Let $\nu = (\nu_1,\dots,\nu_l)$ be a composition of $n$ 
with $\sigma_r = \sigma_s$ for all $\nu_1+\cdots+\nu_{k-1}+1 \leq
r < s \leq \nu_1+\cdots+\nu_k$ and $k=1,\dots,l$. 
Let $\mathfrak{p}_\nu$ be the corresponding standard parabolic
subalgebra of $\g$, i.e. the matrices $A$ and $B$
in (\ref{blockform}) are block upper triangular with diagonal blocks of shape $\nu$.
Let $\sO_\nu$ be the
corresponding parabolic analog of the category $\sO$, i.e. it is the
full subcategory of $\sO$ consisting of all supermodules that are
locally finite-dimensional over $\mathfrak{p}_\nu$.
Similarly, there is a standard parabolic subalgebra $\mathfrak{p}'_\nu$ of
$\g'$ consisting of block upper triangular matrices 
of shape $\nu$, and we let $\sO'_\nu$ be the analogously defined parabolic
subcategory of $\sO'$. 
Various special cases of the following corollary for maximal parabolics/two-part compositions
$\nu$ were known before; see \cite[$\S$4]{C} and \cite{CC}.

\vspace{2mm}
\noindent
{\bf Corollary.} {\em
If $n$ is even then $\sO_\nu$ is superequivalent to $\sO_\nu'$.
If $n$ is odd then $\sO_\nu$ is superequivalent to $(\sO_\nu')^\CT$.
}

\begin{proof}
This follows from our Main Theorem on observing that
$\sO_\nu$ and $\sO_\nu'$ may be
defined equivalently as the Serre subcategories of $\sO$ and $\sO'$
generated by the irreducible supermodules $\{L(\bb), \Pi L(\bb)\}$
and $\{L'(\bb), \Pi L'(\bb)\}$, respectively, for $\bb \in \B$ such that
the following hold for $r \notin\{\nu_1,\nu_1+\nu_2,\dots,\nu_1+\cdots+\nu_l\}$:
\begin{itemize}
\item
if $\sigma_r = +$ then $b_r > b_{r+1}$; 
\item
if $\sigma_r =
-$ then $b_r < b_{r+1}$.
\end{itemize}
(This assertion is a well-known consequence of 
the construction of
parabolic Verma supermodules in $\sO$ and $\sO'$, respectively; see
e.g. \cite{M}.)
\end{proof}

In order to prove the Main Theorem, we will exploit
the following powerful theorem established in \cite{BLW}:
the category $\O'$ defined above 
is the unique (up to strongly equivariant equivalence)
 $\mathfrak{sl}_\infty$-tensor product
categorification of the module 
$$
V^{\otimes \bsigma} := V^{\sigma_1} \otimes
\cdots \otimes V^{\sigma_n},
$$ 
where $V^{+}$ denotes the natural
$\mathfrak{sl}_\infty$-module and $V^{-}$ denotes its dual.
Hence, 
for even $n$,
it suffices to show that the category $\sO$
decomposes as 
$\O \oplus \Pi \O$
for some $\mathfrak{sl}_\infty$-tensor
product categorification $\O$ of $V^{\otimes \bsigma}$.
For odd $n$, we show instead that $\sO^{\CT}$
decomposes as
$\O \oplus \Pi \O$
for some $\mathfrak{sl}_\infty$-tensor
product categorification $\O$ of $V^{\otimes \bsigma}$, 
hence, $\sO^\CT$ is superequivalent to $\sO'$.
The Main Theorem then follows on taking Clifford
twists, using also the lemma formulated above.
In both the even and odd cases, our argument relies crucially also on an
application of the main result of \cite{KKT}.

\vspace{2mm}
\noindent{\em Acknowledgements.}
We thank the authors of \cite{CKW} for several helpful discussions at the
KIAS conference
``Categorical Representation Theory and Combinatorics" in December 2015,
and Shunsuke Tsuchioka for sharing his
computer-generated counterexamples to the type B and C conjectures.

\section{Verma supermodules}\label{s2}
We continue with $n \geq 1$, $\bsigma = (\sigma_1,\dots,\sigma_n) \in
\{\pm 1\}^n$,
and $z \in \K$ with $2z \notin \Z$.
Let $m := \lceil n/2 \rceil$, so that $n=2m$ or $2m-1$.
Also set
\begin{equation}\label{jdef}
I := \Z,
\qquad
J := \left\{\pm \sqrt{z+i}\sqrt{z+i+1}\:\Big|\:i \in I\right\},
\end{equation}
where the square roots are as chosen in the introduction.

In this paragraph, we work with the Lie superalgebra
$\widehat\g := \mathfrak{gl}_{2m|2m}(\K)$ in order to introduce
some coordinates.
Let $\widehat{U}$ be the natural $\widehat\g$-supermodule 
with standard basis
$u_1,\dots,u_{4m}$. Write
$x_{r,s}$ for the $rs$-matrix unit in $\widehat\g$,
so $x_{r,s} u_t = \delta_{s,t} u_r$.
We denote the odd basis vectors 
$u_{2m+1}, \dots, u_{4m}$ instead by $u_1',\dots,u_{2m}'$.
For $1 \leq r,s \leq 2m$, we set
\begin{align}\label{cough1}
e_{r,s} &:= x_{r,s} + x_{2m+r,2m+s},
&
e'_{r,s} &:= x_{r,2m+s}+x_{2m+r,s},\\\label{cough2}
f_{r,s} &:= x_{r,s} - x_{2m+r,2m+s},
&
f'_{r,s} &:=  x_{r,2m+s} - x_{2m+r,s}.\\\intertext{Also let}
d_r &:= e_{r,r},&
d'_r &:= e_{r,r}'.
\end{align}
Then 
we have that
\begin{align}\label{f1}
e_{r,s} u_t &= \delta_{s,t} u_r,
& 
e_{r,s} u_t' &= \delta_{s,t} u_r',   &  e_{r,s}' u_t&= \delta_{s,t}
u_r', & e_{r,s}' u_t' &= \delta_{s,t}
u_r,\\ \label{f2}
f_{r,s} u_t &= \delta_{s,t} u_r, 
& 
f_{r,s} u_t' &= -\delta_{s,t} u_r', & f_{r,s}' u_t&= -\delta_{s,t} u_r', & f_{r,s}' u_t' &= \delta_{s,t}
u_r.
\end{align}
Finally let $\widehat{U}^*$
be the dual supermodule to $\widehat{U}$, with
basis $\phi_1,\dots,\phi_{2m}, \phi_1',\dots,\phi_{2m}'$
that is dual to the basis $u_1,\dots,u_{2m}, u_1',\dots,u_{2m}'$. 
We have that
\begin{align}\label{f3}
e_{r,s} \phi_t &=  - \delta_{r,t} \phi_s, & e_{r,s} \phi_t' &= -\delta_{r,t} \phi_s',&
e_{r,s}' \phi_t &=  - \delta_{r,t} \phi'_s, & e_{r,s}' \phi'_t &=
\delta_{r,t} \phi_s,\\\label{f4}
 f_{r,s} \phi_t &=  -\delta_{r,t} \phi_s, & f_{r,s} \phi_t' &= \delta_{r,t} \phi_s',
&
f_{r,s}' \phi_t &=  -\delta_{r,t} \phi'_s, & f_{r,s}' \phi'_t &=
-\delta_{r,t} \phi_s. 
\end{align}

\vspace{1mm}

When $n$ is even, we continue with 
$\g, \b$ and $\h$ as in the introduction, so $\g$ 
is the subalgebra of $\widehat\g$ spanned by
$\{e_{r,s}, e_{r,s}'\:|\:1 \leq r,s \leq n\}$, while
$\h$ has basis $\{d_r, d_r'\:|\:1 \leq r \leq 2m\}$.
However, when $n$ is odd,
it is convenient 
to change some of this notation.
The point of doing this is to unify our treatment of even and odd $n$ as much
as possible in the remainder of the article.
So, if $n$ is odd, we henceforth {\em redefine} $\g, \b$ and $\h$
as follows:
\begin{itemize}
\item $\g$ denotes the Lie superalgebra $\mathfrak{q}_n(\K) \oplus
  \mathfrak{q}_1(\K)$, which we identify with the subalgebra of
  $\widehat{\g}$
spanned by
$\{e_{r,s}, e_{r,s}'\:|\:1 \leq r,s \leq n\} \sqcup
\{d_{2m}, d_{2m}'\}$.
\item $\b$ is the Borel subalgebra
spanned by
$\{e_{r,s}, e_{r,s}'\:|\:1 \leq r \leq s \leq n\} \sqcup
\{d_{2m}, d_{2m}'\}$;
\item
$\h$ is the Cartan subalgebra
spanned by $\{d_r, d_r'\:|\:1 \leq r \leq 2m\}$.
\end{itemize}
In both the even and the odd cases, 
the subspaces $U \subseteq \widehat{U}$ and $U^* \subseteq \widehat{U}^*$
spanned by $u_1,\dots,u_n, u_1',\dots, u_n'$ 
and $\phi_1,\dots,\phi_n, \phi_1',\dots,\phi_n'$, respectively,
may be viewed as
$\g$-supermodules.
Also set $\t := \h_\0$ and let
$\delta_1,\dots,\delta_{2m}$ be the basis for $\t^*$ that is dual to
the basis $d_1,\dots,d_{2m}$ for $\t$.  

For $\bb \in \B$, we define
$\lambda_\bb$
according to (\ref{wtdict1}) if $n$ is even, but redefine it in the odd
case 
as
\begin{equation}\label{wtdict2}
\lambda_\bb := \sum_{r=1}^n \lambda_{\bb,r} \delta_r  + \delta_{2m}
\quad\text{where}\quad
\lambda_{\bb,r} := \sigma_r(z+b_r).
\end{equation}
We also introduce the tuple $\bd_r \in \B$ 
which has $1$ as its $r$th entry and $0$ in all other places,
so that
\begin{equation}\label{handier}
\lambda_\bb \pm \delta_r 
= \lambda_{\bb \pm \sigma_r \bd_r}.
\end{equation}
Then we define $\sO$ exactly
as we did in the introduction but using the current choices for $\g,
\b, \t$ and $\lambda_\bb$.
This is exactly the same category as in the introduction when $n$
is even, but when $n$ is odd our new version of $\sO$ is superequivalent to the
Clifford twist $\sO^\CT$ of the supercategory from the introduction.
Indeed, if $M$ is a supermodule in our new $\sO$, the restriction of
$M$ to the subalgebra $\q_n(\K)$, equipped with the odd involution
defined by the action of $d_{2m}'$, gives an object of the
Clifford twist of the supercategory from before.

We proceed to define some irreducible $\h$-supermodules
$\{\u(\bb)\:|\:\bb \in \B\}$.
Let $C_2$ be the rank 2 Clifford superalgebra with odd generators
$c_1,c_2$
subject to the relations $c_1^2 = c_2^2 = 1, c_1 c_2 = - c_2 c_1$.
Let $\u$ be the irreducible $C_2$-supermodule
on basis $v, v'$ with $v$ even and $v'$ odd, 
and action defined by $$
c_1 v = v', 
\quad
c_1 v' = v, 
\quad
c_2 v = \sqrt{-1} v', 
\quad
c_2 v' = -
\sqrt{-1} v.
$$
Then, for $\bb \in \B$, we set
$\u(\bb) := \u^{\otimes m}$.
For $1 \leq r \leq n$, we let
$d_r$ act by the scalar $\lambda_{\bb,r}$
and
$d_r'$ act by left
multiplication by
$\sqrt{\lambda_{\bb,r}} \,\id^{\otimes (s-1)}\otimes c_{r+1-2s}
\otimes \id^{ \otimes (m-s)}$
where $s := \lfloor r/2 \rfloor$
(and we are using the usual superalgebra sign rules).
In the odd case, we also need to define the actions of $d_{2m}$ and
$d_{2m}'$: these are the identity and the odd involution
$\id^{\otimes(m-1)} \otimes c_2$, respectively.
In all cases, $\u(\bb)$ is 
an irreducible
$\h$-supermodule of type $\ttM$, and its
$\t$-weight is $\lambda_\bb$.
Moreover, by construction,
$d_1'\cdots d_{2m}'$ acts on any even (resp. odd) vector in $\u(\bb)$
as 
$c_\bb$ (resp. $-c_\bb$), where
\begin{equation}\label{y1}
c_\bb
:=
(\sqrt{-1})^m \sqrt{\la_{\bb,1}}\cdots \sqrt{\la_{\bb,n}}.
\end{equation}
The signs here distinguish $\u(\bb)$ from its parity flip.
The following is well known; e.g. see \cite[$\S$1.5.4]{CW}.

\begin{lemma}\label{projectivity}
For $\bb \in \B$, 
any $\h$-supermodule that is semisimple of weight $\lambda_\bb$ over $\t$
decomposes as a direct sum of copies of the supermodules $\u(\bb)$ and $\Pi \u(\bb)$.
\end{lemma}

\begin{proof}
We can identify $\h$-supermodules that are semisimple of weight $\lambda_\bb$
over $\t$ with supermodules over the Clifford superalgebra
$C_{2m} := C_2^{\otimes m}$, so that
$d_r'\:(r=1,\dots,n)$ acts in the same way as
$\sqrt{\lambda_{\bb,r}} \,
\id^{\otimes (s-1)}\otimes c_{r+1-2s} \otimes \id^{\otimes (m-s)}$
where $s := \lfloor r/2 \rfloor$, and in the odd case $d_{2m}'$ acts as 
$\id^{\otimes(m-1)} \otimes c_2$.
The lemma then follows since
$C_{2m}$ is simple, indeed, it is isomorphic to the matrix superalgebra 
$M_{2^{n-1}|2^{n-1}}(\K)$.
\end{proof}

Let $\underline{\sO}$ denote the {\em underlying category}
consisting of the same
objects as $\sO$ but only the even morphisms.
This is obviously an 
Abelian category.
In order to parametrize its irreducible objects explicitly, 
we introduce the \textit{Verma supermodule} $M(\bb)$ for $\bb \in \B$
by setting
\begin{equation}
M(\bb) := U(\g) \otimes_{U(\b)} \u(\bb),
\end{equation}
where we are viewing $\u(\bb)$ as a $\b$-supermodule by inflating 
along the surjection $\b \twoheadrightarrow \h$.  
The weight $\lambda_\bb$ is the highest weight of $M(\bb)$
in the usual 
\textit{dominance order}
on $\t^*$, i.e. $\lambda \leq \mu$ if and only if $\mu-\lambda\in
\bigoplus_{r=1}^{n-1} \N(\delta_r-\delta_{r+1})$.
Note also that we can distinguish $M(\bb)$ from its parity flip in the same
way as for $\u(\bb)$: the element $d_1'\cdots d_{2m}'$ acts on any
even (resp. odd) vector in the highest weight space 
$M(\bb)_{\lambda_\bb}$ as the scalar
$c_\bb$ (resp. $-c_\bb$).

As usual, the Verma supermodule $M(\bb)$ has a unique
irreducible quotient denoted $L(\bb)$.
Thus, $L(\bb)$ is an irreducible $\g$-supermodule of 
highest weight
$\lambda_\bb$, and the action of $d_1' \cdots
d_{2m}'$ on its highest weight space distinguishes it from its parity
flip.
The irreducible supermodules
$\{L(\bb), \Pi L(\bb)\:|\:\bb \in \B\}$
give a complete set of pairwise inequivalent irreducible supermodules
in $\underline{\sO}$.
The endomorphism algebras of
these objects are all one-dimensional,
so they are irreducibles of type $\mathtt{M}$.
Moreover, 
by a standard argument involving restricting to the underlying even Lie
algebra
as in \cite[Lemma 7.3]{B2},
we get that $\underline{\sO}$ is a {Schurian category}
in the following sense (cf. \cite[$\S$2.1]{BLW}):

\begin{definition}\label{schurcat}
A {\em Schurian category} is a $\K$-linear  Abelian category in which all objects
have finite length, there are enough projectives and injectives,
and the endomorphism algebras of irreducible objects are all one-dimensional.
\end{definition}

Let $x^T$ denote
the usual transpose of a matrix $x \in \widehat{\g}$.  This induces an antiautomorphism
of $\g$, i.e. we have that $[x,y]^T = [y^T, x^T]$.
Given $M \in \ob \sO$, 
we can view the direct sum 
$\bigoplus_{\bb \in \B} M_{\lambda_\bb}^*$ 
of the linear duals of the weight spaces of $M$
as a $\g$-supermodule 
with action defined by $(x f)(v) := f(x^T v)$. 
Let 
$M^\star$ be the object of $\sO$ obtained from this by applying also
the parity switching functor $\Pi^m$.
Making the obvious definition
on morphisms,
this gives us a contravariant superequivalence
$\star : \sO \to \sO$.
We have incorporated the parity flip into this definition
in order to get the following lemma; this is also checked in
\cite[Lemma 7]{F}
or \cite[Lemma 4.7]{C} but we include 
the proof to consolidate our parity conventions.

\begin{lemma}\label{star} For $\bb \in \B$, we have that $L(\bb)^\star \cong
L(\bb)$
via an even isomorphism.
\end{lemma}

\begin{proof}
By weight considerations, we either have that $L(\bb)^\star$ is evenly isomorphic to
$L(\bb)$ or to $\Pi L(\bb)$.
To show that the former holds, 
take an even highest weight vector
$f \in L(\bb)^\star$.
We must show that $d_1'\cdots d_{2m}' f = c_\bb f$ (rather than
$-c_\bb f$).
Remembering the twist by $\Pi^m$ in our definition of $\star$,
there is a highest weight vector $v \in L(\bb)$ of parity $m\pmod{2}$
such that $f(v) = 1$.
Then we get that 
$$
(d_1' \cdots d_{2m}'  f)(v) = 
f(d_{2m}' \cdots d_1' v)
= (-1)^m f(d_1'\cdots d_{2m}' v) = c_\bb f(v).
$$
Hence, $d_1'\cdots d_{2m}' f = c_\bb f$.
\end{proof}

Let $P(\bb)$ be a projective cover of $L(\bb)$ in $\underline{\sO}$.
There are even epimorphisms $P(\bb) \twoheadrightarrow M(\bb)
\twoheadrightarrow L(\bb)$.
Applying $\star$, we deduce that
there are even monomorphisms
$L(\bb) \hookrightarrow M(\bb)^\star \hookrightarrow P(\bb)^\star$.
The supermodule
$P(\bb)^\star$ is an injective hull of $L(\bb)$, while
$M(\bb)^\star$ is the {\em dual Verma supermodule}.
The following lemma is well known; it follows from central
character considerations (e.g. see \cite[Theorem 2.48]{CW}) plus the universal property of Verma supermodules.

\begin{lemma}\label{typicaldominant}
Suppose that $\lambda_{\bb}$ is dominant and typical, by which we mean that
the following hold for all
$1 \leq r < s \leq
n$:
\begin{itemize}
\item
if $\sigma_r = \sigma_s$ then
$\lambda_{\bb,r} \geq \lambda_{\bb,s}$;
\item
if
$\sigma_r \neq \sigma_s$
then
$\lambda_{\bb,r}+\lambda_{\bb,s} \neq 0$.
\end{itemize}
Then $M(\bb) = P(\bb)$.
\end{lemma}

Let $\sO^\Delta$ be the full subcategory of $\sO$ consisting of all
supermodules possessing a Verma flag,
i.e. for which there is a filtration $0 = M_0 \subset \cdots \subset M_l = M$ 
with sections $M_k / M_{k-1}$
that are evenly isomorphic 
to
$M(\bb)$'s or $\Pi M(\bb)$'s for $\bb \in \B$.
Since the classes of all 
$M(\bb)$ and $\Pi M(\bb)$ are linearly independent in the Grothendieck
group of $\underline{\sO}$, the multiplicities $(M:M(\bb))$ and
$(M:\Pi M(\bb))$
of $M(\bb)$ and $\Pi M(\bb)$
in a Verma flag of $M$
are independent of the particular choice of flag.

The following lemma follows from the general theory
developed in \cite[$\S$4]{B2}. We include a self-contained 
proof here in order to make this article independent of \cite{B2}.
Later on, we will also give a self-contained proof of another fundamental
fact established in \cite{B2}, namely, that the projective supermodules
$P(\ba)$ have
Verma flags; cf. Theorem~\ref{mainsplitthm}.

\begin{lemma}\label{hom}
For $M \in \ob\sO^\Delta$ and $\bb \in \B$,
we have that
\begin{align*}
(M:M(\bb)) &= \dim \Hom_{\sO}(M, M(\bb)^\star)_\0,\\
(M:\Pi M(\bb)) &= \dim \Hom_{\sO}(M, M(\bb)^\star)_\1.
\end{align*}
Also, any direct summand of $M \in \ob \sO^\Delta$ possesses a Verma flag.
\end{lemma}

\begin{proof}
The first part of the lemma follows by induction on the length of the
Verma flag, using the following two observations:
for all $\ba,\bb \in \B$ we have that
\begin{itemize}
\item
$\Hom_{\sO}(M(\ba), M(\bb)^\star)$ is zero if $\ba \neq \bb$,
and it is one-dimensional of even parity if $\ba =\bb$;
\item
$\Ext^1_{\sO}(M(\ba), M(\bb)^\star) = 0$.
\end{itemize}
To check these, 
for the first one, we use the universal property of $M(\ba)$ to see
that
$\Hom_{\sO}(M(\ba), M(\bb)^\star)$ is zero unless $\lambda_\ba \leq \lambda_\bb$.
Similarly, on applying $\star$, it is zero unless $\lambda_\bb \leq \lambda_\ba$.
Hence, we may assume that $\ba = \bb$.
Finally, any non-zero homomorphism $M(\ba) \rightarrow M(\ba)^\star$
must send the head to the socle, 
so
$\Hom_{\sO}(M(\ba), M(\ba)^\star)$
is evenly isomorphic to $\Hom_{\sO}(L(\ba), L(\ba)^\star)$,
which is one-dimensional and even thanks to Lemma~\ref{star}.
For the second property,
we must show that all short exact
sequences in $\underline{\sO}$ of the form
$0 \to M(\ba)^\star \to M \to
M(\bb) \to 0$ or $\Pi M(\ba)^\star \to M \to
M(\bb) \to
0$
split.
Either $\lambda_{\ba}$ or $\lambda_{\bb}$ is a maximal weight of $M$.
In the latter case, 
using also Lemma~\ref{projectivity}, we can use the universal property
of $M(\bb)$ to construct a splitting of 
$M \twoheadrightarrow M(\bb)$.
In the former case, we apply $\star$, the
resulting short exact sequence splits as before, and then we dualize again.

The final statement of the lemma may be proved by mimicking the argument for semisimple
Lie algebras from
\cite[\S 3.2]{H}.
\end{proof}

 \section{Special projective superfunctors}\label{s3}

Next, we investigate the superfunctors
$U \otimes -$
and
$U^* \otimes -$
defined by tensoring with the $\g$-supermodules $U$ and $U^*$ introduced in
the previous section.
They clearly preserve the properties of being finitely generated over
$\g$, locally finite-dimensional over $\b$, and semisimple over $\t$.
Since the $\t$-weights of $U$ and $U^*$ are $\delta_1,\dots,\delta_n$
and $-\delta_1,\dots,-\delta_n$, respectively, and
using (\ref{handier}), we get
for each $M \in \ob \sO$ that 
all weights of 
$U \otimes M$ and $U^* \otimes M$
are of the form $\lambda_\bb$ for $\bb \in \B$.
Hence, these superfunctors send objects of $\sO$ to objects of $\sO$,
i.e.
we have defined
\begin{equation}\label{sfdef}
\sF := U \otimes -:\sO \rightarrow \sO,
\qquad
\sE := U^* \otimes -:\sO \rightarrow \sO.
\end{equation}
Recalling the notation (\ref{cough1})--(\ref{cough2}), let
\begin{equation}\label{omegadef}
\omega := 
\sum_{r,s=1}^n \left(f_{r,s}\otimes e_{s,r}  -
  f_{r,s}'\otimes e_{s,r}'\right)
\in U(\widehat{\g}) \otimes U(\g).
\end{equation}
This is essentially the same tensor as in \cite[(7.2.1)]{HKS}.
Left multiplication by $\omega$ (resp.\ by $-\omega$) defines
a linear map
$x_M:U\otimes M \rightarrow U \otimes M$
(resp.\ $x_M^*:U^* \otimes M \rightarrow U^* \otimes M$)
for each $\g$-supermodule $M$.
The following lemma shows that 
these maps 
define a pair of even supernatural transformations
\begin{equation}\label{xdef}
x:\sF \Rightarrow \sF,
\qquad
x^*:\sE \Rightarrow \sE.
\end{equation}
This is also established in the proof of \cite[Theorem 7.4.1]{HKS}, but we give a
self-contained argument since it explains the origin of these maps.

\begin{lemma}
The linear maps $x_M$ and $x_M^*$ just defined are 
even $\g$-supermodule homomorphisms.
\end{lemma}

\begin{proof}
The odd element
\begin{equation}\label{oddz}
f' := \sum_{t=1}^n f_{t,t}'
\in U(\widehat\g)
\end{equation}
supercommutes with the elements of $U(\g)$.
Hence, $f' \otimes 1 \in U(\widehat{\g}) \otimes U(\g)$
supercommutes with 
the image of the comultiplication
$\Delta:U(\g) \rightarrow U(\g) \otimes U(\g) \subset U(\widehat\g) \otimes U(\g)$.
The odd Casimir tensor
$$
\Omega' := 
\sum_{r,s=1}^n 
\left(e_{r,s} \otimes e_{s,r}' - e_{r,s}' \otimes e_{s,r}\right)
\in U(\g) \otimes U(\g)
$$
also supercommutes with the image of $\Delta$.
Hence, the even tensor
$$
\Omega := \Omega' (f' \otimes 1)
=
-\sum_{r,s,t=1}^n \left(e_{r,s}  f_{t,t}'\otimes e_{s,r}' + e_{r,s}'  f_{t,t}'\otimes
  e_{s,r} \right)
\in U(\widehat\g) \otimes U(\g)
$$
commutes with the image of $\Delta$.
Consequently, left multiplication by $\Omega$ defines even
$\g$-supermodule endomorphisms 
$x_M:U \otimes M \rightarrow U \otimes M$
and $x_M^*:U^* \otimes M\rightarrow U^* \otimes M$.
It remains to observe that these endomorphisms agree with the 
linear maps defined by left multiplication by $\omega$ and
$-\omega$, respectively.
Indeed, by a calculation using (\ref{f1})--(\ref{f4}),
the elements  $e_{r,s} f'_{t,t}$ and $e_{r,s}' f_{t,t}'$ of
$U(\widehat\g)$ act on
vectors in $U$ (resp. $U^*$) in the same way as $\delta_{s,t}
f_{r,s}'$  and $-\delta_{s,t} f_{r,s}$
(resp.
$-\delta_{r,t} f_{r,s}'$ and $\delta_{r,t} f_{r,s}$), respectively.
\end{proof}

\begin{lemma} \label{eigenvalues} Suppose that $\bb \in \B$
and let $M := M(\bb)$.
\begin{enumerate}
\item 
There is 
a filtration $$
0 = M_0 \subset M_1 \subset \cdots \subset M_n = 
U \otimes M
$$ 
with
$M_{t} / M_{t-1} \cong M(\bb + \sigma_t \bd_t) \oplus \Pi M(\bb +
\sigma_t \bd_t)$
for each $t=1,\dots,n$.  The endomorphism
$x_{M}$ preserves this filtration,
and the induced endomorphism
of $M_t / M_{t-1}$ is diagonalizable
 with exactly two 
eigenvalues $\pm \sqrt{\lambda_{\bb, t}}\sqrt{\lambda_{\bb,t}+1}$.
Its $\sqrt{\lambda_{\bb, t}}\sqrt{\lambda_{\bb,t} + 1}$-eigenspace is
 evenly isomorphic to $M(\bb + \sigma_t \bd_t)$,
while the other eigenspace is evenly isomorphic to 
$\Pi M(\bb + \sigma_t \bd_t)$.
\item There is 
a filtration 
$$
0 = M^{n} \subset \cdots \subset M^1 \subset M^0 = 
U^*
\otimes  M
$$ 
with
$M^{t-1} / M^{t} \cong M(\bb - \sigma_t \bd_t) \oplus \Pi M(\bb -
\sigma_t \bd_t)$
for each $t=1,\dots,n$.  The endomorphism
$x^*_{M}$ preserves this filtration,
and the induced endomorphism
of $M^{t-1} / M^{t}$ is diagonalizable
 with exactly two 
eigenvalues $\pm \sqrt{\lambda_{\bb, t}}\sqrt{\lambda_{\bb,t}-1}$.
Its $\sqrt{\lambda_{\bb, t}}\sqrt{\lambda_{\bb,t} - 1}$-eigenspace is
 evenly isomorphic to $M(\bb - \sigma_t \bd_t)$,
while the other eigenspace is evenly isomorphic to 
$\Pi M(\bb - \sigma_t \bd_t)$.
 \end{enumerate}
\end{lemma}

\begin{proof}
(1)
The filtration is constructed in \cite[Lemma 4.3.7]{B1}, as follows.
By  the tensor identity
$$
U \otimes M = U \otimes (U(\g) \otimes_{U(\b)} \u(\bb))
			\cong U(\g) \otimes_{U(\b)} (U \otimes \u(\bb)).
$$
As a $\b$-supermodule, $U$ has a filtration 
$0 = U_0 \subset U_1 \subset \cdots\subset U_n = U$
in which the section $U_t / U_{t-1}$ is spanned by the images of
$u_t$ and $u_t'$.  
Let  $M_t$ be the submodule of $U \otimes M$ 
that maps to $U(\g) \otimes_{U(\b)} (U_t \otimes \u(\bb))$
under this isomorphism.

Now fix $t \in \{1,\dots,n\}$.
Let $v_1,\dots,v_k$ be a basis for the even highest weight space
$M(\bb)_{\lambda_\bb,\0}$, so that
$d_t' v_1,\dots,d_t' v_k$ is a 
basis for
$M(\bb)_{\lambda_\bb,\1}$.
The subquotient 
$M_t / M_{t-1} \cong U(\g) \otimes_{U(\b)} (U_t / U_{t-1} \otimes \u(\bb))$
is generated by 
the images of 
the vectors $\{u_t\otimes v_i, u_t' \otimes v_i,
u_t \otimes d_t'v_i, u_t' \otimes d_t'v_i\:|\:
i=1,\dots,k\}$, which by weight considerations
span 
a $\b$-supermodule isomorphic to
$\u(\bb + \sigma_t\bd_t) \oplus \Pi \u(\bb + \sigma_t\bd_t)$.
Hence, $$
M_t / M_{t-1} \cong M(\bb+\sigma_t \bd_t) \oplus \Pi M(\bb+\sigma_t \bd_t).
$$
The action of $f_{r,s} \otimes e_{s,r} - f_{r,s}' \otimes e_{s,r}'$ on
any of 
$u_t \otimes v_i, u_t' \otimes v_i, u_t \otimes d_t'v_i$ or
$u_t'\otimes d_t'v_i$
is zero unless $r \leq s =t$,
and if $r < s = t$ then it sends these vectors into
$M_{t-1}$. Therefore, $x_{M}$ preserves
the filtration. Moreover, this argument shows that it acts on the
highest weight space of the quotient $M_t / M_{t-1}$ in the same way as
$x_t := f_{t,t}\otimes d_t - f_{t,t}'\otimes d'_t$.  

Now consider the purely even subspace $S_{i,t}$ of $M_t / M_{t-1}$
with basis given by the images of $u_t \otimes v_i, u_t' \otimes d_t'v_i$.
Recalling that $d_t$ acts on $v_i$ and on $d_t' v_i$ by $\lambda_{\bb,t}$, and
that $(d_t')^2 = d_t$, it is straightforward to check that the matrix of the
endomorphism $x_t$ of $S_{i,t}$ in the given basis is equal to 
$$
A := \left( \begin{array}{cc} \lambda_{\bb,t} & \lambda_{\bb,t} \\
					1 & -
                                        \lambda_{\bb,t} \end{array}
                                    \right).
$$
Also recall from our construction of $\u(\bb)$
that $d_1'\cdots d_{2m}'$ acts on $v_i$ as the scalar $c_\bb$
from (\ref{y1}), 
and it acts on
$d_t' v_i$ as $-c_{\bb}$.
Using this, another calculation shows that
$d_1'\cdots d_{2m}'$ acts on $S_{i,t}$ as the matrix
$\frac{c_\bb}{\lambda_{\bb,t}} A$.
Similarly, on the purely odd subspace $S_{i,t}'$ 
with basis given by the images of $u_t' \otimes v_i, u_t \otimes d_t'v_i$, $x_t$ has matrix $-A$
and $d_1'\cdots d_{2m}'$ has matrix $-\frac{c_\bb}{\lambda_{\bb,t}} A$.

Since the matrix $A$ has eigenvalues $\pm
\sqrt{\lambda_{\bb,t}}\sqrt{\lambda_{\bb,t}+1}$,
the calculation made in the previous paragraph implies that
$x_t$ is diagonalizable on $M_t / M_{t-1}$
with exactly these eigenvalues.
Moreover on any even highest weight vector in its
$
\sqrt{\lambda_{\bb,t}}\sqrt{\lambda_{\bb,t}+1}$-eigenspace,
we get that $d_1'\cdots d_{2m}'$ acts as
\begin{align*}
		\frac{c_{\bb}}{\lambda_{\bb,t}} \sqrt{\lambda_{\bb,t}}\sqrt{\lambda_{\bb,t} + 1} 
		&= c_{\bb+\sigma_t \bd_t}.
\end{align*}
This implies that the
$\sqrt{\lambda_{\bb,t}}\sqrt{\lambda_{\bb,t}+1}$-eigenspace is
evenly isomorphic to $M(\bb+\sigma_t \bd_t)$.
Similarly,
the
$-\sqrt{\lambda_{\bb,t}}\sqrt{\lambda_{\bb,t}+1}$-eigenspace is
evenly isomorphic to $\Pi M(\bb+\sigma_t \bd_t)$.

(2) Similar.
\end{proof}

\begin{corollary}\label{minpoly}
For $M \in \ob\sO$, 
all roots of the minimal polynomials of
$x_M$ and $x_M^*$ (computed
in the finite dimensional superalgebras
$\End_{\sO}(\sF\,M)$ and $\End_{\sO}(\sE\,M)$)
belong to the set $J$ from (\ref{jdef}).
\end{corollary}

\begin{proof}
This is immediate from the theorem in case $M$ is a Verma supermodule.
We may then deduce that it is true for all irreducibles, hence, for any $M
\in \ob \sO$.
\end{proof}

Corollary~\ref{minpoly} implies that we can decompose
\begin{equation}\label{cc}
\sF = \bigoplus_{j \in J} \sF_j,
\qquad
\sE = \bigoplus_{j \in J} \sE_j,
\end{equation}
where 
$\sF_j$ (resp. $\sE_j$) is the subfunctor of $\sF$ (resp. $\sE$) defined
by letting $\sF_j\, M$
(resp. $\sE_j\, M$) be the generalized $j$-eigenspace of
$x_M$ (resp. $x_M^*$)
for each $M \in \ob\sO$.  
Recall that $I$ denotes $\Z$.
For $i \in I$, we
define the {\em $i$-signature} of $\bb \in \B$
to be the $n$-tuple
$\isig(\bb) = (\isig_1(\bb),\dots,\isig_n(\bb)) \in\{\mathtt{e}, \mathtt{f}, \bullet\}^n$ with
\begin{equation}\label{sigdef}
\isig_t(\bb) := 
\left\{
\begin{array}{ll}
\mathtt{f}&\text{if either }
\sigma_t = +\text{ and }b_t = i,
\text{ or }
\sigma_t = -\text{ and }b_t = i+1,\\
\mathtt{e}&\text{if either }
\sigma_t = + \text{ and }b_t = i+1,
\text{ or }
\sigma_t = -\text{ and }b_t = i,\\
\bullet&\text{otherwise.}
\end{array}
\right.
\end{equation}


\begin{theorem}\label{maintfthm}
Given $\bb \in \B$ and $i \in I$,
let $j := \sqrt{z+i}\sqrt{z+i+1}$.
Then:
\begin{itemize}
\item[(1)]
$\sF_j \,M(\bb)$ has a multiplicity-free 
filtration with sections that are 
evenly isomorphic to the Verma supermodules
$$
\{M(\bb+\sigma_t \bd_t)\:|\:\text{for }1 \leq t \leq n\text{ such that }\isig_t(\bb) =
\mathtt{f}\},
$$
appearing from bottom to top in order of increasing index $t$.
\item[(2)]
$\sE_j \,M(\bb)$ has a multiplicity-free 
filtration with sections that are 
evenly isomorphic to the Verma supermodules
$$
\{M(\bb-\sigma_t \bd_t)\:|\:\text{for }1 \leq t \leq n\text{ such that
}\isig_t(\bb) = \mathtt{e}\},
$$
appearing from top to bottom in order of increasing index $t$.
\end{itemize}
\end{theorem}

\begin{proof}
(1)
It is immediate from Lemma~\ref{eigenvalues} that
$\sF_j \,M(\bb)$ has a multiplicity-free filtration with sections that
are evenly isomorphic to the supermodules
$M(\bb+\sigma_t \bd_t)$
for $t=1,\dots,n$ such that
$\sqrt{\lambda_{\bb,t}}\sqrt{\lambda_{\bb,t}+1} = j$.
Squaring both sides, this equation implies that
$\left(\lambda_{\bb,t}+\half\right)^2 = \left(z+i+\half\right)^2$.
Hence,
$$
\lambda_{\bb,t} = \sigma_t(z+b_t) = 
-\half \pm \left(z+i+\half\right).
$$
We deduce either that $\sigma_t = +$ and $b_t = i$,
or $\sigma_t = -$ and $b_t = i+1$.
Since we squared our original equation,
it remains to check that we do indeed get solutions to
that in both cases. This is clear in the case that $\sigma_t =
+$,
and it follows in the case that $\sigma_t = -$ using also (\ref{violin}).

(2) Similar.
\end{proof}

\begin{remark}
Using Theorem~\ref{hkst} below, one can show that
there are
odd supernatural isomorphisms
$c:\sF_{j}\stackrel{\sim}{\Rightarrow} \sF_{-j}$
and
$c^*:\sE_{j}\stackrel{\sim}{\Rightarrow} \sE_{-j}$
for each $j \in J$.
One consequence (which could be checked directly right away)
is that there is another version of Theorem~\ref{maintfthm}, in which
one takes
 $j := -\sqrt{z+i}\sqrt{z+i+1}$
and replaces the Verma supermodules $M(\bb \pm \sigma_t \bd_t)$ in the statement by
their parity flips.
\end{remark}

The superfunctors $\sF$ and $\sE$ are both left and right adjoint to
each other via some canonical (even) adjunctions.
The adjunction making $(\sE, \sF)$ into an adjoint pair 
is induced by the linear maps
$$
\eps:U^* \otimes U \rightarrow \K,\:
\phi \otimes u \mapsto \phi(u),
\qquad
\eta:\K \rightarrow U \otimes U^*,\:
1 \mapsto \sum_{r=1}^n (u_r \otimes \phi_r + u_r' \otimes \phi_r').
$$
Thus, the unit of adunction $c:1 \Rightarrow \sF \,\sE$
is defined on 
supermodule $M$
by the map
$c_M:M \stackrel{\operatorname{can}}{\longrightarrow} 
\K \otimes M \stackrel{\eta \otimes \id}{\longrightarrow} U \otimes
U^* \otimes M$,
and the counit of adjunction $d:\sE\, \sF \Rightarrow 1$
is defined by $d_M:U^* \otimes U \otimes M
\stackrel{\eps\otimes\id}{\longrightarrow} \K\otimes M
\stackrel{\operatorname{can}}{\longrightarrow} M$.
Similarly, the adjunction making $(\sF, \sE)$ into an adjoint pair
is induced by the linear maps
$$
U \otimes U^* \rightarrow \K,\,
u \otimes \phi \mapsto (-1)^{|\phi||u|}\phi(u),
\quad
\K \rightarrow U^* \otimes U,\,
1 \mapsto \sum_{r=1}^n (\phi_r \otimes u_r - \phi_r' \otimes u_r').
$$
The following lemma implies that these adjunctions restrict to 
adjunctions making $(\sF_j, \sE_j)$ and $(\sE_j, \sF_j)$
into adjoint pairs for each $j \in J$.
It follows that all of these superfunctors send projectives to
projectives,
and they are all exact, i.e. they
preserve short exact sequences
in $\underline{\sO}$.

\begin{lemma}
The supernatural transformation $x^*:\sE \Rightarrow \sE$
is both the
left and right mate of $x:\sF \Rightarrow \sF$
with respect to the canonical adjunctions defined above.
\end{lemma}

\begin{proof}
We just explain how to check that $x^*$ is the left mate of $x$ with respect to the
adjunction $(\sE, \sF)$; the argument for right mate is similar.
We need to show for each $M \in \ob \sO$ that
the composition
$$
U^* \otimes M \stackrel{\operatorname{id} \otimes c_{M}}{\longrightarrow} U^* \otimes U
\otimes U^* \otimes M
\stackrel{\operatorname{id}\otimes x_{U^* \otimes M}}{\longrightarrow}
U^* \otimes U \otimes U^* \otimes M \stackrel{d_{U^* \otimes M}}{\longrightarrow} U^* \otimes M
$$
is equal to $x_M^*:U^*\otimes M \rightarrow U^* \otimes M$.
Recall for this that $x_M^*$ is defined by left multiplication by
$\sum_{r,s=1}^n \left(f_{r,s}' \otimes e_{s,r}' -
  f_{r,s}\otimes e_{s,r}\right)$,
while $x_{U^* \otimes M}$ is defined by left multiplication
by 
$\sum_{r,s=1}^n (f_{r,s} \otimes e_{s,r}\otimes 1 + f_{r,s} \otimes 1 \otimes e_{s,r}
- f_{r,s}' \otimes e_{s,r}'\otimes 1 - f_{r,s}' \otimes 1 \otimes
e_{s,r}')$.
Now one computes the effect of both maps on homogeneous vectors of the
form $\phi_t\otimes v$ and $\phi_t' \otimes v$
using (\ref{f1})--(\ref{f4}).
\end{proof}

\section{Bruhat order}\label{s4}

Consider
the Dynkin
diagram
$
{\begin{picture}(113, 15)%
\put(18,2){\circle{4}}%
\put(37,2){\circle{4}}%
\put(56,2){\circle{4}}%
\put(75, 2){\circle{4}}%
\put(94, 2){\circle{4}}%
\put(20, 2){\line(1, 0){15.5}}%
\put(39, 2){\line(1, 0){15.5}}%
\put(58, 2){\line(1, 0){15.5}}%
\put(77, 2){\line(1, 0){15.5}}%
\put(98, 2){\line(1, 0){1}}%
\put(101, 2){\line(1, 0){1}}%
\put(104, 2){\line(1, 0){1}}%
\put(107, 2){\line(1, 0){1}}%
\put(110, 2){\line(1, 0){1}}%
\put(4, 2){\line(1, 0){1}}%
\put(1, 2){\line(1, 0){1}}%
\put(7, 2){\line(1, 0){1}}%
\put(10, 2){\line(1, 0){1}}%
\put(13, 2){\line(1, 0){1}}%
\put(15, 8){\makebox(0, 0)[b]{$_{{-2}}$}}%
\put(34, 8){\makebox(0, 0)[b]{$_{{-1}}$}}%
\put(56, 8){\makebox(0, 0)[b]{$_{0}$}}%
\put(75, 8){\makebox(0, 0)[b]{$_{1}$}}%
\put(94, 8){\makebox(0, 0)[b]{$_{{2}}$}}%
\end{picture}}
$
whose vertices are indexed by the totally ordered set $I = \Z$.
We denote the associated Kac-Moody algebra 
by
$\ssl_{\infty}$. This is the Lie algebra of traceless,
finitely-supported complex matrices whose rows and columns are indexed
by $I$.
It is generated by the matrix
units $f_i :=e_{i+1, i}$ and $e_i := e_{i, i+1}$ for $i \in I$.  
The \textit{natural representation} $V^{+}$
 of $\mathfrak{sl}_{\infty}$ is the module of 
column vectors with standard basis $\{ v^{+}_i\:|\: i \in I \}$.  
We also need the {\em dual natural representation}
$V^-$ with basis $\{v_i^-\:|\:i \in I\}$.
The action of the Chevalley generators on these bases is given by
\begin{align}\label{e}
e_i v^{+}_j &= \delta_{i+1,j} v^{+}_i,
&
e_i v^{-}_j &= \delta_{i,j} v^{-}_{i+1},\\
f_i v^{+}_j &= \delta_{i,j} v^{+}_{i+1},
& f_i v^{-}_j &= \delta_{i+1,j} v^{-}_i.
\label{f}
\end{align}
The tensor product
$V^{\otimes \bsigma} := V^{\sigma_1} \otimes \cdots \otimes V^{\sigma_n}$
has monomial basis
$\{v_\bb\:|\:\bb \in \B\}$ defined from 
$v_\bb := v^{\sigma_1}_{b_1} \otimes \cdots \otimes v^{\sigma_n}_{ b_n
}$.
Recalling (\ref{sigdef}),
the Chevalley generators act on these monomials by
\begin{align}\label{haha}
f_i v_{\bb} &= \sum_{\substack{1 \leq t \leq n \\ \isig_t(\bb)
    =\mathtt{f}}} v_{\bb+\sigma_t \bd_t},
&
e_i v_{\bb} &= \sum_{\substack{1 \leq t \leq n\\ \isig_t(\bb) =
    \mathtt{e}}} v_{\bb-\sigma_t \bd_t}.
\end{align}
This should be compared with Theorem~\ref{maintfthm}, which already makes some connection between 
the endofunctors $\sF_j, \sE_j$ of
$\sO$ and the $\mathfrak{sl}_\infty$-module $V^{\otimes \bsigma}$.

We next introduce an important partial order $\succeq$ on $\B$,
which we call the {Bruhat order}
It is closely related to the {inverse dominance order} of \cite[Definition 3.2]{LW}, which
comes from Lusztig's construction of tensor products of
based modules \cite[$\S$27.3]{Lubook}.
The root system of $\mathfrak{sl}_\infty$ has 
\textit{weight lattice} $P := \bigoplus_{i \in I} \Z \omega_i$
where $\omega_i$ is the {$i$th
 fundamental weight}.   
For $i \in I$, we set
$$
\eps_i := \omega_i - \omega_{i-1},
\qquad
\alpha_i := \eps_i - \eps_{i+1}.
$$
We identify $\eps_i$ with the weight of the vector $v_i^+$ in the
$\mathfrak{sl}_\infty$-module $V^+$.
Then, $v_i^- \in V^-$ is of weight $-\eps_i$.
For $\bb \in \B$, let
\begin{equation}
\WT(\bb) = (\wt_1(\bb),\dots,\wt_n(\bb)) \in P^n
\end{equation}
be the $n$-tuple of
weights defined from
$\wt_r(\bb) := \sigma_r \eps_{b_r}$,
so that $v_\bb \in V^{\otimes\bsigma}$ is of weight
$|\WT(\bb)| := \wt_1(\bb)+\cdots+\wt_n(\bb) \in P$.
Because the weight spaces of $V^\pm$ are all one-dimensional,
the map $\B \rightarrow P^n, \bb \mapsto \WT(\bb)$ is injective.

\begin{definition}\label{bruhatdef}
Let $\trianglelefteq$ denote the {dominance order} on $P$, so
$\beta \trianglelefteq \gamma\Leftrightarrow\gamma - \beta \in \bigoplus_{i \in
  I} \N \alpha_i$.
The {\em inverse dominance order} on $P^n$
is the partial order defined by declaring that 
$(\beta_1,\dots,\beta_n) \succeq (\gamma_1,\dots,\gamma_n)$
if and only if
$$
\beta_1+\cdots+\beta_s \trianglelefteq \gamma_1+\cdots+\gamma_s,
$$
for each $s=1,\dots,n$, with the inequality being an equality when
$s=n$.
Finally, define the {\em Bruhat order} $\succeq$ on $\B$ 
by $\ba \succeq \bb \Leftrightarrow \WT(\ba) \succeq
\WT(\bb)$.
\end{definition}

Our first lemma makes the definition of the Bruhat order
more explicit.
Using it, one can check in particular that $\ba \succeq \bb$ implies that $\lambda_\ba \geq
\lambda_\bb$ in the dominance order on $\mathfrak{t}^*$; cf. \cite[Lemma 3.4]{BLW}.

\begin{lemma}\label{eleme}
For $\ba \in \B, i \in I$ and $1 \leq s \leq n$, we let
$$
N_{[1,s]}(\ba,i)
:=
\#\{1 \leq r \leq s\:|\:a_r > i, \sigma_r = +\}
-
\#\{1 \leq r \leq s\:|\:a_r > i, \sigma_r = -\}.
$$
Then 
we have that $\ba \succeq \bb$ if and only if
\begin{itemize}
\item
$N_{[1,n]}(\ba, i) = N_{[1,n]}(\bb, i)$
for all $i \in I$;
\item
$N_{[1,s]}(\ba, i) \geq N_{[1,s]}(\bb, i)$
for all $i \in I$ and $s = 1,\dots,n-1$.
\end{itemize}
\end{lemma}

\begin{proof}
This is a special case of
\cite[Lemma 2.17]{BLW}.
\end{proof}

\begin{lemma}\label{not}
Assume that $\ba \succeq \bb$ and 
$\isig_r(\ba) = \isig_n(\bb) =
\mathtt{f}$ 
for some $i \in I$ and
$1 \leq r \leq n$.
Then $\ba + \sigma_r \bd_r \succeq \bb + \sigma_n \bd_n$,
with equality if and only if $\ba = \bb$ and $r=n$.
\end{lemma}

\begin{proof}
We use the conditions from Lemma~\ref{eleme}.
For {\em either }
$j \neq i$ and $1 \leq s \leq n$,
{\em  or} $j = i$ and $1 \leq s < r$,
we have that 
$$
N_{[1,s]}(\ba+\sigma_r \bd_r,j)
=
N_{[1,s]}(\ba,j)
\geq
N_{[1,s]}(\bb,j)
=
N_{[1,s]}(\bb+\sigma_n \bd_n,j).
$$
For $r \leq s < n$, we have that
$$
N_{[1,s]}(\ba+\sigma_r \bd_r, i)
=
N_{[1,s]}(\ba,i)+1
\geq
N_{[1,s]}(\bb,i)+1
>
N_{[1,s]}(\bb,i)
=
N_{[1,s]}(\bb+\sigma_n \bd_n,i).
$$
Finally,
$
N_{[1,n]}(\ba+\sigma_r \bd_r,i)
=
N_{[1,n]}(\ba,i)+1
=
N_{[1,n]}(\bb,i)+1
=
N_{[1,n]}(\bb+\sigma_n \bd_n,i).
$
\end{proof}

To prepare for the next lemma, suppose that we are given $\bb \in \B$.
Define $\ba \in \B$ by setting $a_1 := b_1$, then inductively defining 
$a_s$ 
for $s=2,\dots,n$ as follows.
\begin{itemize}
\item
If $\sigma_s = +$ then 
$a_s$ is the greatest integer such that $a_s \leq b_s$, and the
following hold for all $1 \leq r < s$:
\begin{itemize}
\item[$\circ$]
if $\sigma_r=+$ then
$a_s < a_r$;
\item[$\circ$]
if $\sigma_r=-$
then $a_s < b_r$.
\end{itemize}
\item
If $\sigma_s = -$ then 
$a_s$ is the smallest integer such that $a_s \geq b_s$, and the
following hold for all $1 \leq r < s$:
\begin{itemize}
\item[$\circ$]
if $\sigma_r=-$ then $a_s > a_r$;
\item[$\circ$]
if $\sigma_r=+$ then $a_s > b_r$.
\end{itemize}
\end{itemize}
Also define a
monomial
$X = X_n \cdots X_2$
in the Chevalley generators $\{f_i\:|\:i \in I\}$ by setting
$$
X_r := \left\{
\begin{array}{ll}
f_{b_r-1} \cdots f_{a_r+1} f_{a_r}&\text{if $\sigma_r = +$,}\\
f_{b_r} \cdots f_{a_r-2} f_{a_r-1}&\text{if $\sigma_r = -$,}
\end{array}\right.
$$
for each $r=2,\dots,n$.

\begin{example}
If $\bsigma = (+,+,-,+,-,-)$ and $\bb = (3,4,3,4,3,4)$,
then $\ba = (3,2,5,1,6,7)$ and $X = (f_4 f_5 f_6) (f_3 f_4 f_5) (f_3 f_2
f_1) (f_3 f_4) (f_3 f_2)$.
\end{example}

\begin{lemma}\label{construction}
In the above notation, we have that 
$X v_\ba = v_\bb + 
(\text{a sum of $v_\bc$'s
for $\bc \succ \bb$})$.
\end{lemma}

\begin{proof}
We proceed by induction on $n$, the result being trivial in case $n=1$.
For $n > 1$, let 
$\bar\bsigma :=
(\sigma_1,\dots,\sigma_{n-1})$,
$\bar\ba := (a_1,\dots,a_{n-1})$, $\bar\bb
:= (b_1,\dots,b_{n-1})$
and 
$\bar X := X_{n-1} \cdots X_2$.
Applying the induction hypothesis in the $\mathfrak{sl}_\infty$-module
$V^{\otimes \bar\bsigma}$, we get that
$$
\bar X v_{\bar\ba} = v_{\bar\bb} + 
(\text{a sum of $v_{\bar\bc}$'s
for $\bar\bc \succ \bar\bb$}).
$$
Now we observe that if $f_i$ is a Chevalley generator appearing in one of the monomials $X_r$ for
$r < n$,
then $f_i v^{\sigma_n}_{a_n} = 0$.
This follows from the definitions: 
if $\sigma_n = +$ we must show that
$i \neq a_n$, which follows as $i \geq a_r > a_n$ if $\sigma_r = +$
or $i \geq b_r > a_n$ if $\sigma_r=-$;
if $\sigma_n = -$ we must show that
$i \neq a_n-1$, which follows as $i < b_r < a_n$ if $\sigma_r = +$
or $i < a_r < a_n$ if $\sigma_r=-$.
Hence,
letting
 $\tilde \bb := (b_1,\dots,b_{n-1}, a_n)$, we deduce that
$$
\bar X v_{\ba} = v_{\tilde\bb} + 
(\text{a sum of $v_{\bc}$'s
for $\bc \succ \tilde\bb$}).
$$
Finally we act with $X_n$, which sends $v^{\sigma_n}_{a_n}$ to $v^{\sigma_n}_{b_n}$, and
apply Lemma~\ref{not}.
\end{proof}

\begin{theorem}\label{mainsplitthm} 
For every $\bb \in \B$, the indecomposable projective supermodule
$P(\bb)$ 
has a Verma flag with top section evenly isomorphic to $M(\bb)$
and other
sections 
evenly isomorphic to $M(\bc)$'s for $\bc \in \B$ with $\bc \succ \bb$.
\end{theorem}

\begin{proof}
Let notation be as in 
Lemma~\ref{construction}.
Let $i_1,\dots,i_l \in I$ be defined so that $X$ is the monomial
$f_{i_l} \cdots f_{i_2} f_{i_1}$.
Let $j_k := \sqrt{z+i_k}\sqrt{z+i_k+1}$
for each $k$ and consider the supermodule
$$
P := \sF_{j_l} \cdots \sF_{j_2} \sF_{j_1} M(\ba).
$$
For each $1 \leq r < s \leq n$, we have that
$a_r > a_s$ if $\sigma_r = \sigma_s  = +$,
$a_r < a_s$ if $\sigma_r = \sigma_s  = -$,
and $a_r \neq a_s$ if $\sigma_r \neq \sigma_s$.
This implies that the weight $\lambda_\ba$ is typical and dominant,
hence $M(\ba)$ is projective by Lemma~\ref{typicaldominant}.
Since each $\sF_j$ sends projectives to projectives, 
we deduce that $P$ is projective.
Since the combinatorics of 
(\ref{haha}) matches that of Theorem~\ref{maintfthm},
we can
reinterpret
Lemma~\ref{construction} as saying
that $P$
has a Verma flag with 
one section evenly isomorphic to $M(\bb)$
and all other sections evenly isomorphic to $M(\bc)$'s for $\bc \succ
\bb$.
In fact, 
the unique section isomorphic to $M(\bb)$ appears at the top of this Verma flag, thanks
the order of the sections arising from Theorem~\ref{maintfthm}(1).
Hence, $P$ has a summand evenly isomorphic to
$P(\bb)$, and
it just remains to apply Lemma~\ref{hom}.
\end{proof}

\begin{corollary}\label{bruhat}
For $\bb \in \B$, we have that $[M(\bb):L(\bb)] = 1$. All other composition factors of $M(\bb)$
are evenly isomorphic to $L(\bc)$'s for $\bc \prec \bb$.
\end{corollary}

\begin{proof}
This follows from Theorem~\ref{mainsplitthm} and the 
following analog 
of {\em BGG reciprocity}:
for $\ba, \bb \in \B$, we have that
\begin{align*}
[M(\bb):L(\ba)]
&= \dim \Hom_{\sO}(P(\ba), M(\bb)^\star)_\0
=
(P(\ba):M(\bb)),\\
[M(\bb):\Pi L(\ba)]
&= \dim \Hom_{\sO}(P(\ba), M(\bb)^\star)_\1
=
(P(\ba):\Pi M(\bb)).
\end{align*}
The various equalities here follow from
Lemmas~\ref{hom} and \ref{star}.
\end{proof}

\begin{corollary}\label{split}
For any $\bb \in \B$, every irreducible subquotient of the
indecomposable projective $P(\bb)$ is evenly isomorphic to $L(\ba)$
for $\ba \in \B$ with $|\WT(\ba)| = |\WT(\bb)|$.
\end{corollary}

\begin{proof}
By Theorem~\ref{mainsplitthm}, 
$P(\bb)$ has a Verma flag with sections $M(\bc)$ for $\bc
\succeq \bb$.
By Corollary~\ref{bruhat}, the composition factors of $M(\bc)$ are 
$L(\ba)$'s for $\ba \preceq \bc$.
Hence, every irreducible subquotient of $P(\bb)$
is evenly isomorphic to $L(\ba)$ for $\ba \in \B$ such that $\ba \preceq \bc
\succeq \bb$ for some $\bc$.
This condition implies that $|\WT(\ba)| = |\WT(\bb)|$.
\end{proof}

\section{Weak categorical action}\label{s5}

Let $\O$ be the 
Serre subcategory of $\sO$ generated by $
\{L(\bb)\:|\:\bb \in \B\},$
 i.e. it is the full
subcategory of $\sO$ consisting of all supermodules whose composition
factors are evenly isomorphic to $L(\bb)$'s for $\bb \in \B$.
Since each $L(\bb)$ is of type $\mathtt{M}$, there are no non-zero
odd morphisms between objects of $\O$.
Because of this, we forget the $\Z/2$-grading and simply view
$\O$ as a $\K$-linear category rather than a supercategory.

\begin{theorem}\label{itsplits}
We have that $\sO = \O \oplus \Pi \O$ in the sense defined in the
introduction.
\end{theorem}

\begin{proof}
Let $\Pi \O$ be the
Serre subcategory of $\sO$ generated by 
$\{\Pi L(\ba)\:|\:\ba \in \B\}.$
By Corollary~\ref{split}, all even extensions between $\Pi L(\ba)$ and $L(\bb)$ are split.
Hence, every supermodule in $\sO$ decomposes uniquely as a direct sum
of an object of $\O$ and an object of $\Pi \O$.
The result follows.
\end{proof}

\begin{remark}
For typical blocks,
Theorem~\ref{itsplits} has a more direct proof 
exploiting the 
action of the anticenter of $U(\g)$;
see \cite[$\S$3.1]{F}.
\end{remark}

In order to state our next theorem, 
we briefly recall the following definition due to 
Cline, Parshall and Scott \cite{CPS}:

\begin{definition}\label{hwdef}
A {\em highest weight
  category} is a Schurian category $\C$ 
in the sense of Definition~\ref{schurcat},
together with
an interval-finite poset $(\Lambda, \leq)$ indexing a complete set of
irreducible objects $\{L(\lambda)\:|\:\lambda \in \Lambda\}$, subject
to the following axiom.
For each $\lambda \in \Lambda$, let $P(\lambda)$ be a projective cover of $L(\lambda)$
in $\C$. Define the {\em standard object} $\Delta(\lambda)$ to be the largest
quotient of $P(\lambda)$ such that $[\Delta(\lambda):L(\lambda)] = 1$
and $[\Delta(\lambda):L(\mu)] = 0$ for $\mu \not\leq \lambda$.
Then we require that $P(\lambda)$ has a filtration with top section isomorphic to
$\Delta(\lambda)$
and all other sections of the form $\Delta(\mu)$ for $\mu > \lambda$.
\end{definition}

\begin{theorem}\label{itshw}
The category $\O$ is a highest weight category with
weight poset $(\B, \preceq)$.
Its standard objects are the Verma supermodules $\{M(\bb)\:|\:\bb \in \B\}$.
\end{theorem}

\begin{proof}
It is clear that $\O$ is a Schurian category with isomorphism classes
of irreducible objects represented by $\{L(\bb)\:|\:\bb \in \B\}$.
By Theorem~\ref{mainsplitthm}, 
$P(\bb)$ has a Verma flag with $M(\bb)$ at the top and other sections
that are evenly isomorphic to $M(\bc)$'s for $\bc \succ
\bb$.
It just remains to observe that the Verma supermodules $M(\bb)$
coincide with the standard objects $\Delta(\bb)$. This follows using
the filtration just described plus Corollary~\ref{bruhat}.
\end{proof}

\begin{remark}
By Lemma~\ref{star}, the duality $\star$ on $\sO$ restricts to 
a duality
$\star:\O \rightarrow \O$ fixing isomorphism classes of irreducible
objects.
\end{remark}

Next, take $i \in I$ and set $j := \sqrt{z+i}\sqrt{z+i+1}$.
Theorem~\ref{maintfthm} implies that the exact functors $\sF_j$ and $\sE_j$
send the standard objects in $\O$ to objects of $\O$ 
with a Verma flag.
Hence, they send arbitrary objects in $\O$ to objects of $\O$.
Thus, their restrictions define endofunctors
\begin{equation}\label{them}
F_i := \sF_j|_{\O}:\O\rightarrow\O,
\quad
E_i := \sE_j|_{\O}:\O \rightarrow \O.
\end{equation}
Again, these functors are both left and right adjoint to each other.
Let $\O^\Delta$ be the full subcategory of $\O$ consisting of all
objects possessing a Verma flag.
This is an exact subcategory of $\O$. Its complexified Grothendieck group
$\CC \otimes_{\Z} K_0(\O^\Delta)$
has basis $\{[M(\bb)]\:|\:\bb \in \B\}$.

\begin{theorem}\label{itacts}
For each $i \in I$,
the functors $F_i$ and $E_i$ are exact endofunctors of  $\O^\Delta$.
Moreover, if we identify $\CC \otimes_{\Z} K_0(\O^\Delta)$ with
$V^{\otimes \bsigma}$
so $[M(\bb)] \leftrightarrow v_\bb$ for each $\bb \in \B$, then the induced endomorphisms
$[F_i]$ and $[E_i]$ of the Grothendieck group
act in the same way as the Chevalley generators
$f_i$ and $e_i$ of $\mathfrak{sl}_\infty$.
\end{theorem}

\begin{proof}
Compare
Theorem~\ref{maintfthm} with (\ref{haha}).
\end{proof}

Thus, we have constructed a highest weight category $\O$ with weight
poset $(\B,\preceq)$, and
equipped it with a weak categorical action of the Lie algebra
$\mathfrak{sl}_\infty$
in the sense of \cite{CR, Rou}.

\section{Strong categorical action}\label{s6}

In this section, we upgrade the weak categorical action of
$\mathfrak{sl}_\infty$
on $\mathcal O$ constructed so far to a strong categorical action.
For the following definition, we represent morphisms in a strict monoidal category 
via the usual string calculus, adopting the same conventions for
horizontal and vertical composition
as \cite{KL1}.

\begin{definition}\label{qhdef}
The \textit{quiver Hecke category} 
of type $\mathfrak{sl}_\infty$ 
is the strict $\K$-linear monoidal category 
$\QH$ 
with objects generated by
the set
$I$ from (\ref{jdef}), and morphisms
generated
by
$\mathord{
\begin{tikzpicture}[baseline = -2]
	\draw[-,thick,darkred] (0.08,-.15) to (0.08,.3);
      \node at (0.08,0.05) {$\color{darkred}\bullet$};
   \node at (0.08,-.25) {$\scriptstyle{i}$};
\end{tikzpicture}
}:i \rightarrow i$ and
$\mathord{
\begin{tikzpicture}[baseline = -2]
	\draw[-,thick,darkred] (0.18,-.15) to (-0.18,.3);
	\draw[-,thick,darkred] (-0.18,-.15) to (0.18,.3);
   \node at (-0.18,-.25) {$\scriptstyle{i_2}$};
   \node at (0.18,-.25) {$\scriptstyle{i_1}$};
\end{tikzpicture}
}:i_2 \otimes i_1 \rightarrow i_1 \otimes i_2$,
subject to the following relations:
\begin{align*}
\mathord{
\begin{tikzpicture}[baseline = 2]
	\draw[-,thick,darkred] (0.25,.6) to (-0.25,-.2);
	\draw[-,thick,darkred] (0.25,-.2) to (-0.25,.6);
  \node at (-0.25,-.28) {$\scriptstyle{i_2}$};
   \node at (0.25,-.28) {$\scriptstyle{i_1}$};
      \node at (-0.14,0.42) {$\color{darkred}\bullet$};
\end{tikzpicture}
}
-
\mathord{
\begin{tikzpicture}[baseline =2]
	\draw[-,thick,darkred] (0.25,.6) to (-0.25,-.2);
	\draw[-,thick,darkred] (0.25,-.2) to (-0.25,.6);
  \node at (-0.25,-.28) {$\scriptstyle{i_2}$};
   \node at (0.25,-.28) {$\scriptstyle{i_1}$};
      \node at (0.14,-0.02) {$\color{darkred}\bullet$};
\end{tikzpicture}
}
&=
\mathord{
\begin{tikzpicture}[baseline = 2]
	\draw[-,thick,darkred] (0.25,.6) to (-0.25,-.2);
	\draw[-,thick,darkred] (0.25,-.2) to (-0.25,.6);
  \node at (-0.25,-.28) {$\scriptstyle{i_2}$};
   \node at (0.25,-.28) {$\scriptstyle{i_1}$};
      \node at (-0.13,-0.02) {$\color{darkred}\bullet$};
\end{tikzpicture}
}
-\mathord{
\begin{tikzpicture}[baseline = 2]
	\draw[-,thick,darkred] (0.25,.6) to (-0.25,-.2);
	\draw[-,thick,darkred] (0.25,-.2) to (-0.25,.6);
  \node at (-0.25,-.28) {$\scriptstyle{i_2}$};
   \node at (0.25,-.28) {$\scriptstyle{i_1}$};
      \node at (0.14,0.42) {$\color{darkred}\bullet$};
\end{tikzpicture}
}
=
\left\{
\begin{array}{ll}
\mathord{
\begin{tikzpicture}[baseline = -1]
 	\draw[-,thick,darkred] (0.08,-.3) to (0.08,.4);
	\draw[-,thick,darkred] (-0.28,-.3) to (-0.28,.4);
   \node at (-0.28,-.4) {$\scriptstyle{i_2}$};
   \node at (0.08,-.4) {$\scriptstyle{i_1}$};  
   \end{tikzpicture} 
   }
   & \text{if $i_1 = i_2$,}\\
	\:\:\:0 & \text{if $i_1 \neq i_2$;}\\
\end{array}
\right. \end{align*}\begin{align*}
\mathord{
\begin{tikzpicture}[baseline = 8]
	\draw[-,thick,darkred] (0.28,.4) to[out=90,in=-90] (-0.28,1.1);
	\draw[-,thick,darkred] (-0.28,.4) to[out=90,in=-90] (0.28,1.1);
	\draw[-,thick,darkred] (0.28,-.3) to[out=90,in=-90] (-0.28,.4);
	\draw[-,thick,darkred] (-0.28,-.3) to[out=90,in=-90] (0.28,.4);
  \node at (-0.28,-.4) {$\scriptstyle{i_2}$};
  \node at (0.28,-.4) {$\scriptstyle{i_1}$};
\end{tikzpicture}
}
&=
\left\{
\begin{array}{ll}
\:\:\:0&\text{if $i_1 = i_2$,}\\
(i_2-i_1)
\mathord{
\begin{tikzpicture}[baseline = 0]
	\draw[-,thick,darkred] (0.08,-.3) to (0.08,.4);
	\draw[-,thick,darkred] (-0.28,-.3) to (-0.28,.4);
   \node at (-0.28,-.4) {$\scriptstyle{i_2}$};
   \node at (0.08,-.4) {$\scriptstyle{i_1}$};
     \node at (0.08,0.05) {$\color{darkred}\bullet$};
\end{tikzpicture}
}
+ (i_1-i_2)
\mathord{
\begin{tikzpicture}[baseline = 0]
	\draw[-,thick,darkred] (0.08,-.3) to (0.08,.4);
	\draw[-,thick,darkred] (-0.28,-.3) to (-0.28,.4);
   \node at (-0.28,-.4) {$\scriptstyle{i_2}$};
   \node at (0.08,-.4) {$\scriptstyle{i_1}$};
      \node at (-0.28,0.05) {$\color{darkred}\bullet$};
\end{tikzpicture}
}
&\text{if $|i_1 - i_2| = 1$,}\\
\mathord{
\begin{tikzpicture}[baseline = 0]
	\draw[-,thick,darkred] (0.08,-.3) to (0.08,.4);
	\draw[-,thick,darkred] (-0.28,-.3) to (-0.28,.4);
   \node at (-0.28,-.4) {$\scriptstyle{i_2}$};
   \node at (0.08,-.4) {$\scriptstyle{i_1}$};
\end{tikzpicture}
}&\text{if $|i_1 - i_2| > 1$;}\\
\end{array}
\right. 
\end{align*}\begin{align*}
\mathord{
\begin{tikzpicture}[baseline = 2]
	\draw[-,thick,darkred] (0.45,.8) to (-0.45,-.4);
	\draw[-,thick,darkred] (0.45,-.4) to (-0.45,.8);
        \draw[-,thick,darkred] (0,-.4) to[out=90,in=-90] (-.45,0.2);
        \draw[-,thick,darkred] (-0.45,0.2) to[out=90,in=-90] (0,0.8);
   \node at (-0.45,-.5) {$\scriptstyle{i_3}$};
   \node at (0,-.5) {$\scriptstyle{i_2}$};
  \node at (0.45,-.5) {$\scriptstyle{i_1}$};
\end{tikzpicture}
}
\!\!-
\!\!\!
\mathord{
\begin{tikzpicture}[baseline = 2]
	\draw[-,thick,darkred] (0.45,.8) to (-0.45,-.4);
	\draw[-,thick,darkred] (0.45,-.4) to (-0.45,.8);
        \draw[-,thick,darkred] (0,-.4) to[out=90,in=-90] (.45,0.2);
        \draw[-,thick,darkred] (0.45,0.2) to[out=90,in=-90] (0,0.8);
   \node at (-0.45,-.5) {$\scriptstyle{i_3}$};
   \node at (0,-.5) {$\scriptstyle{i_2}$};
  \node at (0.45,-.5) {$\scriptstyle{i_1}$};
\end{tikzpicture}
}
&=
\left\{
\begin{array}{ll}
(i_2 - i_1) 
\mathord{
\begin{tikzpicture}[baseline = -1]
	\draw[-,thick,darkred] (0.44,-.3) to (0.44,.4);
	\draw[-,thick,darkred] (0.08,-.3) to (0.08,.4);
	\draw[-,thick,darkred] (-0.28,-.3) to (-0.28,.4);
   \node at (-0.28,-.4) {$\scriptstyle{i_3}$};
   \node at (0.08,-.4) {$\scriptstyle{i_2}$};
   \node at (0.44,-.4) {$\scriptstyle{i_1}$};
\end{tikzpicture}
}
&
\text{if $i_1 = i_3$ and $|i_1 - i_2|= 1$,}\\
\:\:\:0 &\text{otherwise.}
\end{array}
\right.
\end{align*}
Let
$I^d$ denote the set of words
$\bi = i_d\cdots i_1$ of length $d$ in the alphabet $I$, 
and identify $\bi \in I^d$ with the object 
$i_d \otimes \cdots \otimes i_1 \in \ob \QH$.
Then, 
the locally unital algebra
\begin{equation}\label{recla}
QH_d := \bigoplus_{\bi,\bi' \in I^d}
\Hom_{\QH}(\bi,\bi')
\end{equation}
is the {\em quiver Hecke algebra} of type $\mathfrak{sl}_\infty$
defined originally by Khovanov and Lauda \cite{KL1} and Rouquier
\cite{Rou}.
\end{definition}

Recall for a $\K$-linear category $\mathcal{C}$
that there is an associated strict $\K$-linear monoidal category
$\mathcal{E}nd(\mathcal{C})$ 
consisting of  $\K$-linear endofunctors and natural transformations.
The remainder of the section will be devoted to the proof of the
following theorem.

\begin{theorem}\label{itlifts}
There is a strict monoidal functor
$\Phi:\QH \rightarrow \mathcal{E}nd(\O)$
sending the generating objects $i \in I$ to the endofunctors
$F_i$ from (\ref{them}).
Moreover,
for all $M \in \ob \O$ and $i \in I$,
the 
endomorphism 
$F_i M \rightarrow F_i M$ defined by
the natural transformation
$\Phi\Big(
\mathord{
\begin{tikzpicture}[baseline = -2]
	\draw[-,thick,darkred] (0.08,-.15) to (0.08,.3);
      \node at (0.08,0.05) {$\color{darkred}\bullet$};
   \node at (0.08,-.25) {$\scriptstyle{i}$};
\end{tikzpicture}
}
\Big)$
is nilpotent.
\end{theorem}

In order to construct $\Phi$, we 
need to pass through two intermediate objects $\AHC$, the (degenerate)
affine Hecke-Clifford supercategory, and $\QHC$, which is a certain
quiver Hecke-Clifford supercategory in the sense of \cite{KKT}.
Both $\AHC$ and $\QHC$ are examples of (strict) {\em monoidal
supercategories}, meaning that they are supercategories equipped
with a monoidal product in 
an appropriate enriched sense. We refer the reader 
to the introduction
of \cite{BE} for the precise definition, just recalling that
morphisms in a monoidal supercategory satisfy the
{\em super interchange law} rather than the usual interchange law of a
monoidal category: in terms of the string calculus as in \cite{BE} we have that
\begin{equation}
\mathord{
\begin{tikzpicture}[baseline = 0]
	\draw[-,thick,darkpurple] (0.08,-.4) to (0.08,-.23);
	\draw[-,thick,darkpurple] (0.08,.4) to (0.08,.03);
      \draw[thick,darkpurple] (0.08,-0.1) circle (4pt);
   \node at (0.08,-0.1) {\color{darkpurple}$\scriptstyle{g}$};
	\draw[-,thick,darkpurple] (-.8,-.4) to (-.8,-.03);
	\draw[-,thick,darkpurple] (-.8,.4) to (-.8,.23);
      \draw[thick,darkpurple] (-.8,0.1) circle (4pt);
   \node at (-.8,.1) {\color{darkpurple}$\scriptstyle{f}$};
\end{tikzpicture}
}
\quad=\quad
\mathord{
\begin{tikzpicture}[baseline = 0]
	\draw[-,thick,darkpurple] (0.08,-.4) to (0.08,-.13);
	\draw[-,thick,darkpurple] (0.08,.4) to (0.08,.13);
      \draw[thick,darkpurple] (0.08,0) circle (4pt);
   \node at (0.08,0) {\color{darkpurple}$\scriptstyle{g}$};
	\draw[-,thick,darkpurple] (-.8,-.4) to (-.8,-.13);
	\draw[-,thick,darkpurple] (-.8,.4) to (-.8,.13);
      \draw[thick,darkpurple] (-.8,0) circle (4pt);
   \node at (-.8,0) {\color{darkpurple}$\scriptstyle{f}$};
\end{tikzpicture}
}
\quad=\quad
(-1)^{|f||g|}\:
\mathord{
\begin{tikzpicture}[baseline = 0]
	\draw[-,thick,darkpurple] (0.08,-.4) to (0.08,-.03);
	\draw[-,thick,darkpurple] (0.08,.4) to (0.08,.23);
      \draw[thick,darkpurple] (0.08,0.1) circle (4pt);
   \node at (0.08,0.1) {\color{darkpurple}$\scriptstyle{g}$};
	\draw[-,thick,darkpurple] (-.8,-.4) to (-.8,-.23);
	\draw[-,thick,darkpurple] (-.8,.4) to (-.8,.03);
      \draw[thick,darkpurple] (-.8,-0.1) circle (4pt);
   \node at (-.8,-.1) {\color{darkpurple}$\scriptstyle{f}$};
\end{tikzpicture}
}
\end{equation}
for homogeneous morphisms $f$ and $g$ of parities $|f|$ and $|g|$, respectively.

\begin{definition}
The (degenerate) {\em affine Hecke-Clifford supercategory} $\AHC$
is the strict 
monoidal supercategory
with a single generating object $1$, even generating morphisms
$\mathord{
\begin{tikzpicture}[baseline = -1]
	\draw[-,thick,darkblue] (0.08,-.15) to (0.08,.3);
      \node at (0.08,0.08) {$\color{darkblue}\bullet$};
\end{tikzpicture}
}:1 \rightarrow 1$ and
$\mathord{
\begin{tikzpicture}[baseline = -1]
	\draw[-,thick,darkblue] (0.18,-.15) to (-0.18,.3);
	\draw[-,thick,darkblue] (-0.18,-.15) to (0.18,.3);
\end{tikzpicture}
}:1 \otimes 1 \rightarrow 1 \otimes 1$,
and an odd generating morphism
$\mathord{
\begin{tikzpicture}[baseline = -1]
	\draw[-,thick,darkblue] (0.08,-.15) to (0.08,.3);
      \node at (0.08,0.08) {$\color{darkblue}\circ$};
\end{tikzpicture}
}:1 \rightarrow 1$.
These are subject to the following relations:
\begin{align*}
\mathord{
\begin{tikzpicture}[baseline = 2]
	\draw[-, thick,darkblue] (0, 0.6) to (0, -0.3);
	\node at (0, -0.00) {$\color{darkblue} \circ$};
	\node at (0, 0.3) {$\color{darkblue} \bullet$};
\end{tikzpicture}
} &= 
- 
\mathord{
\begin{tikzpicture}[baseline = 2]
	\draw[-, thick,darkblue] (0, 0.6) to (0, -0.3);
	\node at (0, -0.0) {$\color{darkblue} \bullet$};
	\node at (0, 0.3) {$\color{darkblue} \circ$};
\end{tikzpicture}
}\:,
&\mathord{
\begin{tikzpicture}[baseline = 2]
	\draw[-, thick,darkblue] (0, 0.6) to (0, -0.3);
	\node at (0, -0.0) {$\color{darkblue} \circ$};
	\node at (0, 0.3) {$\color{darkblue} \circ$};
\end{tikzpicture}
}
&= 
\mathord{
\begin{tikzpicture}[baseline = 2]
	\draw[-, thick,darkblue] (0, 0.6) to (0, -0.3);
\end{tikzpicture}
}\:,
&
\mathord{
\begin{tikzpicture}[baseline = 8.5]
	\draw[-,thick,darkblue] (0.2,.4) to[out=90,in=-90] (-0.2,.9);
	\draw[-,thick,darkblue] (-0.2,.4) to[out=90,in=-90] (0.2,.9);
	\draw[-,thick,darkblue] (0.2,-.1) to[out=90,in=-90] (-0.2,.4);
	\draw[-,thick,darkblue] (-0.2,-.1) to[out=90,in=-90] (0.2,.4);
\end{tikzpicture}}
 &= 
\mathord{
\begin{tikzpicture}[baseline = 8.5]
	\draw[-,thick,darkblue] (0.1,-.1) to (0.1,.9);
	\draw[-,thick,darkblue] (-0.3,-.1) to (-0.3,.9);
\end{tikzpicture}
}\:,
\end{align*}\begin{align*}
\mathord{
\begin{tikzpicture}[baseline = 4]
	\draw[-,thick,darkblue] (0.25,.6) to (-0.25,-.2);
	\draw[-,thick,darkblue] (0.25,-.2) to (-0.25,.6);
      \node at (-0.14,0.42) {$\color{darkblue}\circ$};
\end{tikzpicture}
}
 &= 
\mathord{
\begin{tikzpicture}[baseline = 4]
	\draw[-,thick,darkblue] (0.25,.6) to (-0.25,-.2);
	\draw[-,thick,darkblue] (0.25,-.2) to (-0.25,.6);
      \node at (0.14,-0.02) {$\color{darkblue}\circ$};
\end{tikzpicture}
}
&
\mathord{
\begin{tikzpicture}[baseline = 4]
	\draw[-,thick,darkblue] (0.25,.6) to (-0.25,-.2);
	\draw[-,thick,darkblue] (0.25,-.2) to (-0.25,.6);
      \node at (-0.14,0.42) {$\color{darkblue}\bullet$};
\end{tikzpicture}
}
-
\mathord{
\begin{tikzpicture}[baseline = 4]
	\draw[-,thick,darkblue] (0.25,.6) to (-0.25,-.2);
	\draw[-,thick,darkblue] (0.25,-.2) to (-0.25,.6);
      \node at (0.14,-0.02) {$\color{darkblue}\bullet$};
\end{tikzpicture}
}
&= 
\mathord{
\begin{tikzpicture}[baseline = 4]
 	\draw[-,thick,darkblue] (0.08,-.2) to (0.08,.6);
	\draw[-,thick,darkblue] (-0.28,-.2) to (-0.28,.6);
\end{tikzpicture}
} 
- \!
\mathord{
\begin{tikzpicture}[baseline = 4]
 	\draw[-,thick,darkblue] (0.08,-.2) to (0.08,.6);
	\draw[-,thick,darkblue] (-0.28,-.2) to (-0.28,.6);
	\node at (0.08, 0.2) {$\color{darkblue}\circ$};
	\node at (-0.28, 0.2){$\color{darkblue}\circ$};
\end{tikzpicture}
}\:,
&
\mathord{
\begin{tikzpicture}[baseline = 3]
	\draw[-,thick,darkblue] (0.45,.8) to (-0.45,-.4);
	\draw[-,thick,darkblue] (0.45,-.4) to (-0.45,.8);
        \draw[-,thick,darkblue] (0,-.4) to[out=90,in=-90] (-.45,0.2);
        \draw[-,thick,darkblue] (-0.45,0.2) to[out=90,in=-90] (0,0.8);
\end{tikzpicture}
}
&=
\mathord{
\begin{tikzpicture}[baseline = 3]
	\draw[-,thick,darkblue] (0.45,.8) to (-0.45,-.4);
	\draw[-,thick,darkblue] (0.45,-.4) to (-0.45,.8);
        \draw[-,thick,darkblue] (0,-.4) to[out=90,in=-90] (.45,0.2);
        \draw[-,thick,darkblue] (0.45,0.2) to[out=90,in=-90] (0,0.8);
\end{tikzpicture}
}\,.
\end{align*}
Denoting the object $1^{\otimes d} \in \ob \AHC$
simply by $d$,
the (degenerate) {\em affine Hecke-Clifford superalgebra}
is the superalgebra
\begin{equation}\label{train1}
AHC_d := \End_{\AHC}(d).
\end{equation}
This was introduced originally by Nazarov \cite[$\S$3]{N}.
\end{definition}

For a supercategory $\C$,
we write $\mathcal{E}nd(\C)$
for the strict monoidal 
supercategory 
consisting of superfunctors 
and supernatural transformations.

\begin{theorem}\label{hkst}
There is a strict monoidal superfunctor
$\Psi:\AHC \rightarrow \mathcal{E}nd(\sO)$
sending the generating object
$1$ to the endofunctor $\sF = U \otimes -$ from (\ref{sfdef}), and the generating morphisms 
$\mathord{
\begin{tikzpicture}[baseline = -1]
	\draw[-,thick,darkblue] (0.08,-.15) to (0.08,.3);
      \node at (0.08,0.08) {$\color{darkblue}\bullet$};
\end{tikzpicture}
}$,
$\mathord{
\begin{tikzpicture}[baseline = -1]
	\draw[-,thick,darkblue] (0.08,-.15) to (0.08,.3);
      \node at (0.08,0.08) {$\color{darkblue}\circ$};
\end{tikzpicture}
}$ and
$\mathord{
\begin{tikzpicture}[baseline = -1]
	\draw[-,thick,darkblue] (0.18,-.15) to (-0.18,.3);
	\draw[-,thick,darkblue] (-0.18,-.15) to (0.18,.3);
\end{tikzpicture}
}$
to the supernatural transformations
$x, c$ and $t$ which are
 defined on $M \in \ob \sO$ as follows:
\begin{itemize}
\item 
$x_M:U \otimes M \rightarrow U \otimes M$ 
is left multiplication by the tensor $\omega$ from (\ref{omegadef});
\item
$c_M: U \otimes M \rightarrow U \otimes M$
is left  multiplication by $\sqrt{-1} \,f' \otimes 1$
for $f'$ as in (\ref{oddz});
\item
$t_M:U \otimes U \otimes M \rightarrow U \otimes U \otimes M$ 
sends
$u \otimes v \otimes m \mapsto (-1)^{|u||v|} v \otimes u
\otimes m$.
\end{itemize}
\end{theorem}

\begin{proof}
This an elementary check of relations, similar to the one made in the
proof of
\cite[Theorem 7.4.1]{HKS}.
\end{proof}

\begin{definition}\label{qhcdef} The \textit{quiver Hecke-Clifford supercategory}
of type $\mathfrak{sl}_\infty$ is the monoidal supercategory 
$\QHC$ with 
objects generated by the set $J$ from (\ref{jdef}), even generating morphisms 
$\mathord{
\begin{tikzpicture}[baseline = -2]
	\draw[-,thick,darkgreen] (0.08,-.15) to (0.08,.3);
      \node at (0.08,0.05) {$\color{darkgreen}\bullet$};
   \node at (0.08,-.25) {$\scriptstyle{j_1}$};
\end{tikzpicture}
}:j_1 \rightarrow j_1$ and
$\mathord{
\begin{tikzpicture}[baseline = -2]
	\draw[-,thick,darkgreen] (0.18,-.15) to (-0.18,.3);
	\draw[-,thick,darkgreen] (-0.18,-.15) to (0.18,.3);
   \node at (-0.18,-.25) {$\scriptstyle{j_2}$};
   \node at (0.18,-.25) {$\scriptstyle{j_1}$};
\end{tikzpicture}
}:j_2 \otimes j_1 \rightarrow j_1 \otimes j_2$,
and odd generating morphisms
$\mathord{
\begin{tikzpicture}[baseline = -2]
	\draw[-,thick,darkgreen] (0.08,-.15) to (0.08,.3);
      \node at (0.08,0.05) {$\color{darkgreen}\circ$};
   \node at (0.08,-.25) {$\scriptstyle{j_1}$};
\end{tikzpicture}
}:j_1 \rightarrow -j_1$,
for all $j_1,j_2 \in J$. These are subject to the following relations:
\begin{align*}
\mathord{
\begin{tikzpicture}[baseline = 3]
	\draw[-, thick,darkgreen] (0, 0.6) to (0, -0.2);
	\node at (0, 0.05) {$\color{darkgreen} \bullet$};
	\node at (0, 0.35) {$\color{darkgreen} \circ$};
	\node at (-0.01,-.35) {$\scriptstyle{j_1}$};
\end{tikzpicture}
}
&=
- 
\mathord{
\begin{tikzpicture}[baseline = 3]
	\draw[-, thick,darkgreen] (0, 0.6) to (0, -0.2);
	\node at (0, 0.05) {$\color{darkgreen} \circ$};
	\node at (0, 0.35) {$\color{darkgreen} \bullet$};
	\node at (-0.01,-.35) {$\scriptstyle{j_1}$};
\end{tikzpicture}
},
\qquad
\mathord{
\begin{tikzpicture}[baseline = 3]
	\draw[-, thick,darkgreen] (0, 0.6) to (0, -0.2);
	\node at (0, 0.05) {$\color{darkgreen} \circ$};
	\node at (0, 0.35) {$\color{darkgreen} \circ$};
	\node at (-0.01,-.35) {$\scriptstyle{j_1}$};
\end{tikzpicture}
}
=
\mathord{
\begin{tikzpicture}[baseline = 3]
	\draw[-, thick,darkgreen] (0, 0.6) to (0, -0.2);
	\node at (-0.01,-.35) {$\scriptstyle{j_1}$};
\end{tikzpicture}
},
\qquad
\mathord{
\begin{tikzpicture}[baseline = 4]
	\draw[-,thick,darkgreen] (0.25,.6) to (-0.25,-.2);
	\draw[-,thick,darkgreen] (0.25,-.2) to (-0.25,.6);
      \node at (-0.13,-0.02) {$\color{darkgreen}\circ$};
        \node at (-0.25,-.28) {$\scriptstyle{j_2}$};
   \node at (0.25,-.28) {$\scriptstyle{j_1}$};
\end{tikzpicture}
}
=
\mathord{
\begin{tikzpicture}[baseline = 4]
	\draw[-,thick,darkgreen] (0.25,.6) to (-0.25,-.2);
	\draw[-,thick,darkgreen] (0.25,-.2) to (-0.25,.6);
      \node at (0.14,0.42) {$\color{darkgreen}\circ$};
        \node at (-0.25,-.28) {$\scriptstyle{j_2}$};
   \node at (0.25,-.28) {$\scriptstyle{j_1}$};
\end{tikzpicture}
},
\qquad
\mathord{
\begin{tikzpicture}[baseline = 4]
	\draw[-,thick,darkgreen] (0.25,.6) to (-0.25,-.2);
	\draw[-,thick,darkgreen] (0.25,-.2) to (-0.25,.6);
      \node at (-0.13,0.42) {$\color{darkgreen}\circ$};
        \node at (-0.25,-.28) {$\scriptstyle{j_2}$};
   \node at (0.25,-.28) {$\scriptstyle{j_1}$};
\end{tikzpicture}
} = 
\mathord{
\begin{tikzpicture}[baseline = 4]
	\draw[-,thick,darkgreen] (0.25,.6) to (-0.25,-.2);
	\draw[-,thick,darkgreen] (0.25,-.2) to (-0.25,.6);
      \node at (0.14,-0.02) {$\color{darkgreen}\circ$};
        \node at (-0.25,-.28) {$\scriptstyle{j_2}$};
   \node at (0.25,-.28) {$\scriptstyle{j_1}$};
\end{tikzpicture}
},\end{align*}\begin{align*}
\mathord{
\begin{tikzpicture}[baseline = 1]
	\draw[-,thick,darkgreen] (0.25,.6) to (-0.25,-.2);
	\draw[-,thick,darkgreen] (0.25,-.2) to (-0.25,.6);
  \node at (-0.25,-.28) {$\scriptstyle{j_2}$};
   \node at (0.25,-.28) {$\scriptstyle{j_1}$};
      \node at (-0.13,-0.02) {$\color{darkgreen}\bullet$};
\end{tikzpicture}
}
-\mathord{
\begin{tikzpicture}[baseline = 1]
	\draw[-,thick,darkgreen] (0.25,.6) to (-0.25,-.2);
	\draw[-,thick,darkgreen] (0.25,-.2) to (-0.25,.6);
  \node at (-0.25,-.28) {$\scriptstyle{j_2}$};
   \node at (0.25,-.28) {$\scriptstyle{j_1}$};
      \node at (0.14,0.42) {$\color{darkgreen}\bullet$};
\end{tikzpicture}
}
&=
\left\{
\begin{array}{cl}
\mathord{
\begin{tikzpicture}[baseline = -5]
 	\draw[-,thick,darkgreen] (0.08,-.3) to (0.08,.4);
	\draw[-,thick,darkgreen] (-0.28,-.3) to (-0.28,.4);
   \node at (-0.28,-.4) {$\scriptstyle{j_2}$};
   \node at (0.08,-.4) {$\scriptstyle{j_1}$};  
   \end{tikzpicture} 
   }
   & \text{if $j_1 = j_2$,} \\
   \mathord{
\begin{tikzpicture}[baseline = -2]
 	\draw[-,thick,darkgreen] (0.08,-.3) to (0.08,.4);
	\draw[-,thick,darkgreen] (-0.28,-.3) to (-0.28,.4);
	\node at (-0.28, 0.06){$\color{darkgreen}\circ$};
	\node at (0.08, 0.06){$\color{darkgreen}\circ$};
   \node at (-0.28,-.4) {$\scriptstyle{j_2}$};
   \node at (0.08,-.4) {$\scriptstyle{j_1}$};
\end{tikzpicture}
 }
	& \text{if $j_1= - j_2$,} \\
	0 & \text{otherwise;}
\end{array}
\right.\end{align*}\begin{align*}
%
%
\mathord{
\begin{tikzpicture}[baseline = 1]
	\draw[-,thick,darkgreen] (0.25,.6) to (-0.25,-.2);
	\draw[-,thick,darkgreen] (0.25,-.2) to (-0.25,.6);
  \node at (-0.25,-.28) {$\scriptstyle{j_2}$};
   \node at (0.25,-.28) {$\scriptstyle{j_1}$};
      \node at (-0.13,0.42) {$\color{darkgreen}\bullet$};
\end{tikzpicture}
}
-
\mathord{
\begin{tikzpicture}[baseline = 1]
	\draw[-,thick,darkgreen] (0.25,.6) to (-0.25,-.2);
	\draw[-,thick,darkgreen] (0.25,-.2) to (-0.25,.6);
  \node at (-0.25,-.28) {$\scriptstyle{j_2}$};
   \node at (0.25,-.28) {$\scriptstyle{j_1}$};
      \node at (0.15,-0.02) {$\color{darkgreen}\bullet$};
\end{tikzpicture}
}
&=
\left\{
\begin{array}{cl}
\mathord{
\begin{tikzpicture}[baseline = -5]
 	\draw[-,thick,darkgreen] (0.08,-.3) to (0.08,.4);
	\draw[-,thick,darkgreen] (-0.28,-.3) to (-0.28,.4);
   \node at (-0.28,-.4) {$\scriptstyle{j_2}$};
   \node at (0.08,-.4) {$\scriptstyle{j_1}$};  
   \end{tikzpicture} 
   }
   & \text{if $j_1 = j_2$,} \\
   \!\!\!\!-\mathord{
\begin{tikzpicture}[baseline = -2]
 	\draw[-,thick,darkgreen] (0.08,-.3) to (0.08,.4);
	\draw[-,thick,darkgreen] (-0.28,-.3) to (-0.28,.4);
	\node at (-0.28, 0.06){$\color{darkgreen}\circ$};
	\node at (0.08, 0.06){$\color{darkgreen}\circ$};
   \node at (-0.28,-.4) {$\scriptstyle{j_2}$};
   \node at (0.08,-.4) {$\scriptstyle{j_1}$};
\end{tikzpicture}
 }
	& \text{if $j_1= - j_2$,} \\
	0 & \text{otherwise;}
\end{array}
\right.\end{align*}\begin{align*}
\mathord{
\begin{tikzpicture}[baseline = 7]
	\draw[-,thick,darkgreen] (0.28,.4) to[out=90,in=-90] (-0.28,1.1);
	\draw[-,thick,darkgreen] (-0.28,.4) to[out=90,in=-90] (0.28,1.1);
	\draw[-,thick,darkgreen] (0.28,-.3) to[out=90,in=-90] (-0.28,.4);
	\draw[-,thick,darkgreen] (-0.28,-.3) to[out=90,in=-90] (0.28,.4);
  \node at (-0.28,-.4) {$\scriptstyle{j_2}$};
  \node at (0.28,-.4) {$\scriptstyle{j_1}$};
\end{tikzpicture}
}
&=
\left\{
\begin{array}{ll}
\:\:\:\:0&\text{if $i_1 = i_2$,}\\
\mathord{
\kappa_1 (i_1-i_2)
\begin{tikzpicture}[baseline = -1]
	\draw[-,thick,darkgreen] (0.08,-.3) to (0.08,.4);
	\draw[-,thick,darkgreen] (-0.28,-.3) to (-0.28,.4);
   \node at (-0.28,-.4) {$\scriptstyle{j_2}$};
   \node at (0.08,-.4) {$\scriptstyle{j_1}$};
     \node at (0.08,0.05) {$\color{darkgreen}\bullet$};
\end{tikzpicture}
}
+
\kappa_2(i_2-i_1)
\mathord{
\begin{tikzpicture}[baseline = -1]
	\draw[-,thick,darkgreen] (0.08,-.3) to (0.08,.4);
	\draw[-,thick,darkgreen] (-0.28,-.3) to (-0.28,.4);
   \node at (-0.28,-.4) {$\scriptstyle{j_2}$};
   \node at (0.08,-.4) {$\scriptstyle{j_1}$};
      \node at (-0.28,0.05) {$\color{darkgreen}\bullet$};
\end{tikzpicture}
}
&\text{if $|i_1 - i_2| = 1$,}\\
\mathord{
\begin{tikzpicture}[baseline = 0]
	\draw[-,thick,darkgreen] (0.08,-.3) to (0.08,.4);
	\draw[-,thick,darkgreen] (-0.28,-.3) to (-0.28,.4);
   \node at (-0.28,-.4) {$\scriptstyle{j_2}$};
   \node at (0.08,-.4) {$\scriptstyle{j_1}$};
\end{tikzpicture}
}&\text{if $|i_1 - i_2| > 1$;}\\
\end{array}
\right. 
\end{align*}\begin{align*}
\mathord{
\begin{tikzpicture}[baseline = 1]
	\draw[-,thick,darkgreen] (0.45,.8) to (-0.45,-.4);
	\draw[-,thick,darkgreen] (0.45,-.4) to (-0.45,.8);
        \draw[-,thick,darkgreen] (0,-.4) to[out=90,in=-90] (-.45,0.2);
        \draw[-,thick,darkgreen] (-0.45,0.2) to[out=90,in=-90] (0,0.8);
   \node at (-0.45,-.5) {$\scriptstyle{j_3}$};
   \node at (0,-.5) {$\scriptstyle{j_2}$};
  \node at (0.45,-.5) {$\scriptstyle{j_1}$};
\end{tikzpicture}
}
-
\mathord{
\begin{tikzpicture}[baseline = 1]
	\draw[-,thick,darkgreen] (0.45,.8) to (-0.45,-.4);
	\draw[-,thick,darkgreen] (0.45,-.4) to (-0.45,.8);
        \draw[-,thick,darkgreen] (0,-.4) to[out=90,in=-90] (.45,0.2);
        \draw[-,thick,darkgreen] (0.45,0.2) to[out=90,in=-90] (0,0.8);
   \node at (-0.45,-.5) {$\scriptstyle{j_3}$};
   \node at (0,-.5) {$\scriptstyle{j_2}$};
  \node at (0.45,-.5) {$\scriptstyle{j_1}$};
\end{tikzpicture}
}
\!&=
\left\{
\begin{array}{ll}
\kappa_1(i_1 - i_2) 
\mathord{
\begin{tikzpicture}[baseline = -1]
	\draw[-,thick,darkgreen] (0.44,-.3) to (0.44,.4);
	\draw[-,thick,darkgreen] (0.08,-.3) to (0.08,.4);
	\draw[-,thick,darkgreen] (-0.28,-.3) to (-0.28,.4);
   \node at (-0.28,-.4) {$\scriptstyle{j_3}$};
   \node at (0.08,-.4) {$\scriptstyle{j_2}$};
   \node at (0.44,-.4) {$\scriptstyle{j_1}$};
\end{tikzpicture}
}
&
\text{if $j_1 = j_3$ and $|i_1 - i_2|= 1$,}\\
\kappa_1 (i_2 - i_1) 
\mathord{
\begin{tikzpicture}[baseline =-1]
	\draw[-,thick,darkgreen] (0.44,-.3) to (0.44,.4);
	\draw[-,thick,darkgreen] (0.08,-.3) to (0.08,.4);
	\draw[-,thick,darkgreen] (-0.28,-.3) to (-0.28,.4);
	\node at (-0.28, 0.06) {$\color{darkgreen} \circ$};
	\node at (0.44, 0.06) {$\color{darkgreen} \circ$};
   \node at (-0.28,-.4) {$\scriptstyle{j_3}$};
   \node at (0.08,-.4) {$\scriptstyle{j_2}$};
   \node at (0.44,-.4) {$\scriptstyle{j_1}$};
\end{tikzpicture}
}
&\text{if $j_1=-j_3$ and $|i_1 - i_2| = 1$,}
\\
\:\:\:\:0 &\text{otherwise.}
\end{array}
\right.
\end{align*}
In the above, we have adopted the convention given $j_r \in J$ 
that $i_r \in I$
and $\kappa_r \in \{\pm 1\}$ are defined from
$j_r = \kappa_r \sqrt{z+i_r}\sqrt{z + i_r +1}$.
Identifying the word $\bj = j_d\cdots j_1 \in J^d$ with
$j_d \otimes\cdots\otimes j_1 \in \ob \QHC$, 
the {\em quiver Hecke-Clifford superalgebra} is the locally
unital algebra
\begin{equation}\label{train2}
QHC_d :=
\bigoplus_{\bj,\bj' \in J^d}
\Hom_{\QHC}(\bj,\bj').
\end{equation}
This is exactly as in \cite[Definition 3.5]{KKT} in the special case of the
$\mathfrak{sl}_\infty$-quiver.
\end{definition}

Now we 
are going to exploit 
a remarkable isomorphism between certain completions $\widehat{AHC}_d$ and
$\widehat{QHC}_d$ of the superalgebras $AHC_d$ and $QHC_d$ 
from (\ref{train1}) and (\ref{train2}), which was
constructed in \cite{KKT}.
To define these, we need some
further notation. 

Numbering 
strands of a diagram by $1,\dots,d$ from right to left,
$AHC_d$ is generated by its elements
$x_r, c_r\:(1 \leq r \leq d)$ and $t_r\:(1 \leq r < d)$
corresponding
to the closed dot on the $r$th strand,
the open dot on the $r$th strand, and the crossing of the $r$th and
$(r+1)$th strands, respectively.
Let $HC_d := S_d \ltimes C_d$ be the 
{\em Sergeev superalgebra}, that is, the smash product
of the symmetric group $S_d$ with basic transpositions
$t_1,\dots,t_{d-1}$
acting on the Clifford superalgebra $C_d$ on generators $c_1,\dots,c_d$.
Let $A_d$ denote
the 
purely even polynomial superalgebra 
$\K[x_1,\dots,x_d]$.
By the basis theorem for $AHC_d$ established in \cite[$\S$2-k]{BK}, 
the natural multiplication map gives a superspace isomorphism
$HC_d \otimes A_d \stackrel{\sim}{\rightarrow}
AHC_d$.
Transporting the multiplication on $AHC_d$ to $HC_d\otimes A_d$ via this
isomorphism, 
the following describe how to commute a polynomial
$f \in A_d$ past the
generators of $HC_d$:
\begin{align}
  (1 \otimes f) (c_r \otimes 1) &= c_r \otimes c_r(f),\\
 (1 \otimes f) (t_r \otimes 1) &= t_r \otimes t_r(f) + 1 \otimes \partial_r(f)
+c_r c_{r+1} \otimes \tilde\partial_r(f),
\end{align}
for operators $c_r, t_r, \partial_r,\tilde\partial_r:A_d \rightarrow
A_d$ 
such that
\begin{itemize}
\item $t_r$ is the automorphism that interchanges $x_r$ and
  $x_{r+1}$ and
fixes all other generators;
\item
$c_r$ is the automorphism that sends $x_r \mapsto -x_r$ and
fixes all other generators;
\item
$\partial_r$ is the Demazure operator $\partial_r(f) :=
\frac{t_r(f)-f}{x_r - x_{r+1}};
$
\item
$\tilde{\partial}_r$ is the {twisted Demazure operator} $c_{r+1} \circ \partial_r \circ c_{r}$,
so
$\tilde\partial_r(f) =
\frac{t_r(f) - c_{r+1}(c_r(f))}{x_r+x_{r+1}}.$
\end{itemize}
 Given a tuple $\mu = (\mu_i)_{i \in I}$ of non-negative integers all
but finitely many of which are zero, 
the quotient superalgebra
\begin{equation}
AHC_d(\mu) := AHC_d \Big/ \Big\langle
\prod_{i \in I} \left(x_1^2 - (z+i)(z+i+1)\right)^{\mu_i}
\Big\rangle
\end{equation}
is a (degenerate) {\em cyclotomic Hecke-Clifford superalgebra} in the
sense of
\cite[$\S$3.e]{BK}.
It is finite dimensional. Moreover, all roots of the minimal polynomials
of
all $x_r \in AHC_d(\mu)$ belong to the set $J$.
It follows for each $\bj = j_d \cdots j_1$ in the set $J^d$ of words
of length $d$ in letters $J$
that there is
an idempotent
$1_\bj \in AHC_d(\mu)$ defined by the projection onto the simultaneous
generalized eigenspaces for $x_1,\dots,x_d$ with eigenvalues
$j_1,\dots,j_d$, respectively.
Moreover, we have that
$$
AHC_d(\mu) = \bigoplus_{\bj,\bj' \in J^d} 1_{\bj'} AHC_d(\mu) 1_\bj.
$$
If $\mu \leq \mu'$, i.e. $\mu_i \leq \mu_i'$ for all $i$,
there is a canonical surjection 
$AHC_d(\mu') \twoheadrightarrow AHC_d(\mu)$
sending $x_r,c_r,t_r,1_\bj \in AHC_d(\mu')$ to the elements of
$AHC_d(\mu)$ with the same names.
Let
\begin{equation}\label{phew}
\widehat{AHC}_d := \varprojlim_{\mu} AHC_d(\mu)
\end{equation}
be the inverse limit of this system of superalgebras
taken in the category of locally
unital superalgebras with distinguished idempotents indexed by $J^d$.
Using the basis theorem for the cyclotomic quotients $AHC_d(\mu)$
from \cite[$\S$3-e]{BK}, one can identify $\widehat{AHC}_d$ 
with the completion defined in \cite[Definition
5.3]{KKT}\footnote{Note there is a sign error in \cite[(5.5)]{KKT}:
it should read $-C_a C_{a+1}\dots$.}.
In particular, letting
$$
\widehat{A}_d := \bigoplus_{\bj \in J^d} 
\K[[x_1-j_1,\dots,x_d-j_d]]1_\bj,
$$
there is a superspace isomorphism
$HC_d \otimes \widehat{A}_d \stackrel{\sim}{\rightarrow}
\widehat{AHC}_d$
induced by the obvious multiplication maps
$HC_d \otimes \widehat{A}_d \twoheadrightarrow AHC_d(\mu)$
for all $\mu$.
The multiplication on
$HC_d\otimes \widehat{A}_d$
corresponding to the one on $\widehat{AHC}_d$ via this isomorphism
has 
the following properties for all
$f \in \widehat{A}_d$:
\begin{align}
 (1 \otimes f 1_{\bj}) (c_r \otimes 1_{\bj'})
&= c_r \otimes c_r(f)  1_{c_r(\bj)} 1_{\bj'},\\\notag
(1 \otimes f 1_{\bj}) (t_r \otimes 1_{\bj'}) &=  t_r \otimes t_r(f) 1_{t_r(\bj)}1_\bj' 
+ 1 \otimes \frac{t_r(f) 1_{t_r(\bj)} 
- f 1_\bj}{x_r-x_{r+1}}1_{\bj'} \\
&+c_r c_{r+1} \otimes \frac{
t_r(f)  1_{t_r(\bj)}- c_{r+1}(c_{r}(f))  1_{c_{r+1}(c_r(\bj))}}{x_r + x_{r+1}} 1_{\bj'}.
\label{fr}\end{align}
The fractions on the right hand side of (\ref{fr})
make sense:
in the first,
$(x_r - x_{r+1}) 1_{\bj'}$ is
invertible unless $j'_r = j'_{r+1}$, 
in which case the expression 
equals $\partial_r(f) 1_\bj 1_{\bj'}$;
the second is fine when $j'_r \neq -j'_{r+1}$ as then $(x_r+x_{r+1}) 1_{\bj'}$ is
invertible, while if $j'_r = -j'_{r+1}$ it equals $\tilde\partial_r(f) 1_{t_r(\bj)}1_{\bj'}$.

Similarly, there is a completion $\widehat{QHC}_d$ of $QHC_d$.
To introduce this,
we denote the elements
of $QHC_d 1_\bj$ defined by an open dot on the $r$th strand,
a closed dot on the $r$th strand and a crossing of the $r$th and
$(r+1)$th strands by
$\gamma_r 1_\bj, \xi_r 1_\bj$ and $\tau_r 1_\bj$, respectively.
For $\mu = (\mu_i)_{i \in I}$ as above, we define the
{\em cyclotomic quiver Hecke-Clifford superalgebra}
\begin{equation}
QHC_d(\mu) := QHC_d \Big/ \left\langle
\xi_1^{2\mu_{i}}1_\bj \:\Big|\:\bj \in  J^d
, i \in I
\text{ with }j_1^2 = (z+i)(z+i+1)
\right\rangle.
\end{equation}
Using the relations, it is easy to see that the images of all $\xi_r 1_\bj$ are nilpotent
in $QHC_d(\mu)$.
Then we set
\begin{equation}
\widehat{QHC}_d := \varprojlim_\mu QHC_d(\mu),
\end{equation}
taking the inverse limit once again in the category of locally unital
superalgebras with distinguished idempotents indexed by $J^d$.
The obvious locally unital homomorphisms
$QHC_d
\otimes_{\K[\xi_1,\dots,\xi_d]} \K[[\xi_1,\dots,\xi_d]]
\twoheadrightarrow QHC_d(\mu)$
for each $\mu$ induce a surjective homomorphism
$$
QHC_d
\otimes_{\K[\xi_1,\dots,\xi_d]} \K[[\xi_1,\dots,\xi_d]]
\rightarrow
\widehat{QHC}_d.
$$
This map is actually an isomorphism, as may be deduced using the basis
theorem for $QHC_d$ from \cite[Corollary 3.9]{KKT}
plus the observation that the image of any non-zero element $u \in QHC_d$ is
non-zero in $QHC_d(\mu)$ for sufficiently large $\mu$; the latter
assertion follows by elementary considerations involving the natural
$\Z$-grading on $QHC_d$.
Consequently, $\widehat{QHC}_d$ is
isomorphic to the completion introduced in a slightly different way in
\cite[Definition 3.16]{KKT}.
Moreover, 
there is a locally unital embedding
$QHC_d \hookrightarrow \widehat{QHC}_d$. 

At last, we are ready to state 
the crucial theorem from \cite{KKT}.
We need this only 
 in the special situation of
\cite[$\S$5.2(i)(a)]{KKT}, but emphasize that
the results obtained in \cite{KKT} are substantially more
general.
In particular, for us, all elements of the set $I$ are
even in the sense of \cite[$\S$3.5]{KKT}, so that we do not
need the more general quiver Hecke {\em super}algebras of \cite{KKT}.

\begin{theorem}\label{kkttheorem}
There is a superalgebra isomorphism
$\widehat{QHC}_d \stackrel{\sim}{\rightarrow} \widehat{AHC}_d$
such that 
$$
1_\bj \mapsto 1_\bj,\quad
\gamma_r 1_\bj \mapsto c_r 1_\bj,\quad
\xi_r 1_\bj  \mapsto y_r 1_\bj,\quad
\tau_r 1_\bj  \mapsto t_r g_r 1_\bj
+ f_r 1_\bj
+
c_r c_{r+1} \tilde f_r 1_\bj,
$$
for all $\bj \in J^d$
and 
$r$.
Here, $y_r \in \K[[x_r-j_r]]$ and $g_r, f_r, \tilde f_r \in \K[[x_r-j_r,x_{r+1}-j_{r+1}]]$ are
the power series determined uniquely by the following:
\begin{align*}
j_r &= \kappa_r \sqrt{z+i_r}\sqrt{z+i_r +1}\text{ for $i_r \in I$ and $\kappa_r \in \{\pm\}$},\\
y_r &=\kappa_r \left(\sqrt{x_r^2+{\textstyle\frac{1}{4}}} - \left(z+i_r+\half\right)\right) \in (x_r-j_r),\\
p_r &= \frac{(x_r^2-x_{r+1}^2)^2}{2 (x_r^2+x_{r+1}^2)-(x_r^2-x_{r+1}^2)^2},\\
g_r &=
\left\{
\begin{array}{ll}
-1&\text{if $i_r < i_{r+1}$,}\\
p_r
\left(\kappa_r y_r - \kappa_{r+1} y_{r+1}\right)
&\text{if $i_r = i_{r+1}+1$,}\\
p_r&\text{if $i_r > i_{r+1}+1$,}\\
\frac{\sqrt{p_r}}{y_r-y_{r+1}}
\in \frac{x_r-x_{r+1}}{y_r-y_{r+1}} + (x_r-x_{r+1})
&\text{if $j_r = j_{r+1}$,}\\
\frac{\sqrt{p_r}}{y_r+y_{r+1}}
\in \frac{x_r+x_{r+1}}{y_r+y_{r+1}} + (x_r+x_{r+1})
&\text{if $j_r = -j_{r+1}$;}\\
\end{array}\right.\\
f_r &= \frac{g_r}{x_r-x_{r+1}} -
  \frac{\delta_{j_r,j_{r+1}}}{y_r-y_{r+1}},\qquad
\tilde f_r = \frac{g_r}{x_r+x_{r+1}} -
  \frac{\delta_{j_r,-j_{r+1}}}{y_r+y_{r+1}}.
\end{align*}
(All of this notation depends implicitly on $\bj$.)
\end{theorem}

\begin{proof}
This is a special case of \cite[Theorem 5.4]{KKT}. 
To help the reader to translate between our notation and that of \cite{KKT}, we note that
the set $J$ in \cite{KKT} is the same as our set $J$, but the set $I$
there is $\tilde I := \{j^2\:|\:j \in J\}$, which is different from
our $I$.
We have 
made various other choices as stipulated in \cite{KKT} in order to produce
concrete formulae:
we have taken the functions $\eps:J \rightarrow \{0,1\}$ and $h:\tilde I
\rightarrow \K$ from \cite[(5.7)]{KKT} so that 
$\eps(j) = (1-\kappa)/2$ 
and $h(j^2) = z+i+\half$
for $J\ni j = \kappa \sqrt{z+i}\sqrt{z+i+1}$;
for \cite[(5.11)]{KKT} we took $G_{j_r, j_{r+1}}$ (our $g_r$) to be $-1$ when 
$i_r < i_{r+1}$.
The fact that $g_r, f_r$ and $\tilde f_r$ are all well-defined elements of
$\K[[x_r-j_r,x_{r+1}-j_{r+1}]]$
is justified by \cite[Lemma 5.5]{KKT}.
Note also that 
the ambiguous square roots appearing in the formulae for $y_r$ and
$g_r$ are uniquely
determined by the containments we have specified.
\end{proof}

\begin{proof}[Proof of Theorem~\ref{itlifts}]
For $\bi = i_d \cdots i_1 \in I^d$, let $F_\bi := F_{i_d} \cdots
F_{i_1}:\O\rightarrow\O$.
The usual vertical composition of natural transformations makes the
vector space
$$
NT_d := \bigoplus_{\bi,\bi' \in I^d} \Hom(F_\bi, F_{\bi'})
$$
into a locally unital algebra with distinguished
idempotents 
$\{1_\bi\:|\:\bi \in I^d\}$ arising from the identity endomorphisms of
each $F_\bi$.
Also horizontal composition of natural transformations defines
homomorphisms
$a_{d_2,d_1}:NT_{d_2} \otimes NT_{d_1} \rightarrow NT_{d_2+d_1}$
for all $d_1,d_2 \geq 0$.
Recalling (\ref{recla}), the data of a strict monoidal functor $\Phi:\QH \rightarrow
\mathcal{E}nd(\O)$ sending $i$ to $F_i$
is just the same as a family of 
locally unital algebra homomorphisms
$\Phi_d:
QH_d \rightarrow NT_d$
for all $d \geq 0$, such that 
$1_\bi \mapsto 1_\bi$ for each $\bi \in I^d$ and
\begin{equation}\label{analog}
a_{d_2,d_1} \circ \Phi_{d_2} \otimes \Phi_{d_1} = \Phi_{d_2+d_1} \circ
b_{d_2,d_1}
\end{equation}
for all $d_1,d_2 \geq 0$,
where $b_{d_2,d_1}:QH_{d_2} \otimes QH_{d_1} \rightarrow QH_{d_2+d_1}$
is the obvious embedding defined by horizontal concatenation of
diagrams.

To construct $\Phi_d$, we start from 
the monoidal superfunctor $\Psi$ from Theorem~\ref{hkst}.
This induces
superalgebra homomorphisms
$\Psi_d:AHC_d \rightarrow \End(\sF^d)$ 
for all $d \geq 0$,
where
$\End(\sF^d)$ denotes 
supernatural endomorphisms of $\sF^d:\sO \rightarrow \sO$.
For each $M \in \ob \sO$,
Corollary~\ref{minpoly} implies that
$\operatorname{ev}_M \circ \Psi_d : AHC_d \rightarrow \End_{\sO}(\sF^d \,M)$
factors through all sufficiently
large cyclotomic quotients
$AHC_d(\mu)$. 
Hence, $\Psi_d$ extends uniquely to a locally unital superalgebra homomorphism
$\widehat{\Psi}_d:\widehat{AHC}_d \rightarrow SNT_d$,
where
$$
SNT_d := \bigoplus_{\bj, \bj' \in J^d} \Hom(\sF_\bj, \sF_{\bj'}) \subset \End(\sF^d)
$$
and $\sF_{\bj} := \sF_{j_d}\cdots \sF_{j_1}$.
Composing $\widehat{\Psi}_d$ with the isomorphism from
Theorem~\ref{kkttheorem}
and the inclusion $QHC_d \hookrightarrow \widehat{QHC}_d$, we obtain a
locally unital superalgebra homomorphism
$\Theta_d:QHC_d \rightarrow SNT_d$.
It is obvious from Definitions~\ref{qhdef} and \ref{qhcdef}
that there is a locally unital algebra homomorphism
$\inc:QH_d \rightarrow (QHC_d)_\0$
sending the idempotent $1_\bi$ to $1_\bj$
for $\bj$ with $j_r := \sqrt{z+i_r}\sqrt{z+i_r+1}$,
and taking the elements of $QH_d 1_\bi$
defined by the dot on the $r$th strand and the crossing of the $r$th
and $(r+1)$th strands
to $\xi_r 1_\bj$ and $\tau_r 1_\bj$, respectively.
Also, recalling (\ref{them}), 
restriction from $\sO$ to $\O$
defines a homomorphism
$$
\pr:\displaystyle\bigoplus_{\bj, \bj' \in J_+^d} 1_{\bj'} (SNT_d)_\0 1_\bj
\rightarrow NT_d
$$
where $J_+ := \left\{\sqrt{z+i}\sqrt{z+i+1}\:\big|\:i \in I\right\} \subset J$.
Then the composition $\pr \circ \Theta_d \circ \inc$ gives us the
desired locally unital homomorphism $\Phi_d:QH_d \rightarrow NT_d
$
sending $1_\bi \mapsto 1_\bi$ for each $\bi \in I^d$.
It just remains to observe that the property (\ref{analog}) is
satisfied,
and that $\Phi_d(x_r 1_\bi)_M$ is nilpotent for each $r$,
$\bi \in I^d$
and $M \in \ob\O$.
These things follow from the explicit formulae in Theorems~\ref{hkst} and \ref{kkttheorem} plus
Corollary~\ref{minpoly} once again.
\end{proof}

\section{Proof of the Main Theorem}

Everything is now in place for us to be able to prove the Main Theorem
from the introduction. Recall $\bsigma = (\sigma_1,\dots,\sigma_n)$ is
a sign sequence, and $V^{\otimes \bsigma}$ denotes the
$\mathfrak{sl}_\infty$-module
$V^{\sigma_1}\otimes\cdots\otimes V^{\sigma_n}$.
The following is a special case of \cite[Definition 2.10]{BLW}, which reformulated
\cite[Definiton 3.2]{LW} for tensor products of minuscule
representations;
it may be helpful to recall Definitions~\ref{bruhatdef}, \ref{hwdef}
and \ref{qhdef}
at this point.

\begin{definition}\label{tpcdef}
An $\mathfrak{sl}_\infty$-{\em tensor product categorification} of
$V^{\otimes \bsigma}$
is the following data:
\begin{itemize}
\item
a highest weight category $\mathcal C$ 
with weight poset $(\B, \preceq)$;
\item
adjoint pairs $(F_i, E_i)$ of endofunctors of $\mathcal C$ for each $i
\in I$;
\item
a strict monoidal functor $\Phi:\mathcal{QH} \rightarrow
\mathcal{E}nd(\C)$ with $\Phi(i) = F_i$ for each $i \in I$.
\end{itemize}
We impose the following additional axioms for all $i
\in I$, $\bb \in \B$ and $M \in \ob \C$:
\begin{itemize}
\item
$E_i$ is isomorphic to a left adjoint
of $F_i$;
\item
$F_i \Delta(\bb)$ 
has a $\Delta$-flag with sections 
$\{\Delta(\bb+\sigma_t \bd_t)\:|\: 1 \leq t \leq n,
\isig_t(\bb) = \mathtt{f}\}$;
\item
$E_i \Delta(\bb)$ 
has a $\Delta$-flag with sections 
$\{\Delta(\bb-\sigma_t \bd_t)\:|\:1 \leq t \leq n, \isig_t(\bb) =
\mathtt{e}\}$;
\item
the endomorphism
$\Phi\Big(
\mathord{
\begin{tikzpicture}[baseline = -2]
	\draw[-,thick,darkred] (0.08,-.15) to (0.08,.3);
      \node at (0.08,0.05) {$\color{darkred}\bullet$};
   \node at (0.08,-.25) {$\scriptstyle{i}$};
\end{tikzpicture}
}
\Big)_M:
F_i M\rightarrow F_i M$
is nilpotent.
\end{itemize}
\end{definition}

Theorems~\ref{itsplits}, \ref{itshw}, \ref{itacts} and \ref{itlifts}
together imply:

\begin{theorem}\label{typea}
The supercategory $\sO$ defined in section~\ref{s2} splits as $\O \oplus \Pi \O$, 
with $\O$
admitting all of the additional structure needed to make it into an
$\mathfrak{sl}_\infty$-tensor product categorification of $V^{\otimes \bsigma}$.
\end{theorem}

As we already mentioned in the introduction, 
to complete the proof of our Main Theorem,
we
just need to appeal to the following results from \cite{BLW}:

\begin{theorem}
The supercategory $\sO'$ from the introduction decomposes as $\O'
\oplus \Pi \O'$, with $\O'$ admitting the structure of an
$\mathfrak{sl}_\infty$-tensor product categorification of
$V^{\otimes \bsigma}$.
\end{theorem}

\begin{proof}
This is a special case of \cite[Theorem 3.10]{BLW}.
\end{proof}

\begin{theorem}
Any two $\mathfrak{sl}_\infty$-tensor product categorifications of
$V^{\otimes \bsigma}$ are strongly equivariantly equivalent
(in the sense of \cite[Definition 3.1]{LW}).
\end{theorem}

\begin{proof}
This is a special case of \cite[Theorem 2.12]{BLW}.
\end{proof}

We get at once that the categories $\O$ and $\O'$ are equivalent,
hence so too are $\sO$ and $\sO'$. This already proves the Main
Theorem from the introduction in the case that $n$ is even. When $n$
is odd, one also needs to apply the Lemma from the introduction to see that
the supercategory $\sO$ in the statement of the Main Theorem is the Clifford
twist of the supercategory $\sO$ being studied here.

\section{Canonical basis} \label{cbs}

Combining our Main Theorem with the results of \cite{CLW, BLW}, it
follows that the composition
multiplicities
$[M(\ba):L(\bb)]$ can be obtained by evaluating certain
parabolic Kazhdan-Lusztig polynomials at $q=1$.
In this section, we explain a simple algorithm to compute these
polynomials explicitly.
The algorithm 
is similar in spirit to the algorithm explained in \cite[$\S$2-j]{B0}, but actually the
variant here is both easier
to implement and a little faster. It also has the
advantage of working
for arbitrary sign sequences $\bsigma$, whereas the approach in
\cite{B0} only makes sense for {\em normally-ordered} $\bsigma$'s,
i.e. ones in which
all $+$'s preceed all $-$'s. (But note that one can easily transition between different sign sequences 
as explained in \cite[$\S$5]{CL}.)

We first need to introduce the quantum analog of the
$\mathfrak{sl}_\infty$-module $V^{\otimes\bsigma}$. We will
use similar notation to before, 
but decorated with dots to indicate $q$-analogs.
Consider the generic quantized enveloping algebra $U_q \mathfrak{sl}_\infty$
over the field $\Q(q)$.
This has standard generators 
$\{\dot e_i, \dot f_i, \dot k_i, \dot k_i^{-1}\:|\:i
\in I\}$. We work with the comultiplication $\Delta$ defined from
\begin{equation}\label{comult}
\Delta(\dot{f}_i) = 
1 \otimes \dot{f}_i  + \dot{f}_i \otimes \dot{k}_i,
\quad
\Delta(\dot{e}_i) = 
\dot{k}_i^{-1} \otimes \dot{e}_i  + \dot{e}_i \otimes 1,
\quad
\Delta(\dot{k}_i) = \dot{k}_i \otimes \dot{k}_i.
\end{equation}
We have the natural $U_q \mathfrak{sl}_\infty$-module $\dot V^+$ on
basis $\{\dot v^+_i\:|\:i \in I\}$, and its dual $\dot V^-$ on basis
$\{\dot v^-_i\:|\:i \in I\}$.
The Chevalley generators $\dot f_i$ and $\dot e_i$ 
act on these basis vectors by exactly the same formulae (\ref{e})--(\ref{f}) 
as at $q=1$, and also
\begin{equation}
\dot k_i \dot v^+_j = q^{\delta_{i,j}-\delta_{i+1,j}} \dot v^+_j,
\qquad
\dot k_i \dot v^-_j = q^{\delta_{i+1,j}-\delta_{i,j}}\dot v^-_j.
\end{equation}
Then we form the tensor space $\dot V^{\otimes \bsigma} := \dot V^{\sigma_1}
\otimes\cdots\otimes \dot V^{\sigma_n}$,
which is a $U_q \mathfrak{sl}_\infty$-module with its monomial basis 
$\left\{\dot v_\bb := \dot v_{b_1}^{\sigma_1} \otimes\cdots\otimes
  \dot v_{b_n}^{\sigma_n}\:\big|\:\bb \in \B\right\}$.

Next we pass to a certain
completion 
$\widehat{V}^{\otimes\bsigma}$.
Recall the inverse dominance order $\succeq$ on $P^n$ from
Definition~\ref{bruhatdef}.
We define $\widehat{V}^{\otimes\bsigma}$
to be the $\Q(q)$-vector space
consisting of formal linear combinations of the form
$\sum_{\bb \in \B} p_\bb(q) \dot v_\bb$ 
for rational functions $p_\bb(q) \in \Q(q)$, such that 
the {\em support} $\{\bb \in \B\:|\:p_\bb(q) \neq 0\}$
is contained in a finite union of sets of the form
$
\{\bb \in \B\:|\:\WT(\bb) \succeq \bbeta\}
$
for $\bbeta \in P^n$.

\begin{lemma}
The action of 
$U_q
\mathfrak{sl}_\infty$
on $\dot V^{\otimes\bsigma}$ extends to a well-defined action on
$\widehat{V}^{\otimes\bsigma}$
such that $u \big(\sum_{\bb \in \B} p_\bb(q) \dot v_\bb\big) = 
\sum_{\bb \in \B} p_\bb(q) \,u \dot v_\bb$
for every $u \in U_q \mathfrak{sl}_\infty$.
Moreover,
$\widehat{V}^{\otimes \bsigma}$ splits as the direct
sum of its weight spaces.
\end{lemma}

\begin{proof}
For the first assertion, we need to show that the expression
$\sum_{\bb \in \B} p_\bb(q) \,u \dot v_\bb$ satisfies the condition on
its support required to belong to $\widehat{V}^{\otimes\bsigma}$.
It suffices to check this for $u \in \{\dot f_i, \dot e_i\:|\:i \in I\}$.
If $\WT(\bb) \succeq (\beta_1,\dots,\beta_n)$ and $v_\ba$ appears with non-zero
coefficient in the expansion of
$\dot e_i v_\bb$ (resp. $\dot f_i v_\bb$),
then $\WT(\ba)$ is equal to $\WT(\bb)$ with $\alpha_i$ added
(resp. subtracted) from one of its entries.
Hence, $\WT(\ba) \succeq (\beta_1+\alpha_i,\beta_2,\dots,\beta_n)$
(resp. $\WT(\ba) \succeq
(\beta_1,\dots,\beta_{n-1},\beta_n-\alpha_i)$).
This is all that is needed.

The fact that $\widehat{V}^{\otimes\bsigma}$ is a direct sum of its weight spaces 
follows because all $v_\bb$ with $\WT(\bb) \succeq (\beta_1,\dots,\beta_n)$ are of the same
weight $\beta_1+\cdots+\beta_n$.
\end{proof}

The completion $\widehat{V}^{\otimes \bsigma}$ admits a
{\em bar involution}
$\psi: \widehat{V}^{\otimes\bsigma}
\rightarrow 
\widehat{V}^{\otimes\bsigma}$ which is anti-linear with respect to
the field automorphism 
$\Q(q) \rightarrow \Q(q), q \mapsto q^{-1}$.
To define $\psi$, let $\Theta$ be Lusztig's quasi-$R$-matrix
from \cite[Theorem 4.1.2]{Lubook}; 
note for
this due to our different choice of $\Delta$ compared to \cite{Lubook}
that Lusztig's $v$ is our $q^{-1}$ (and his $E_i, F_i, K_i$ are our
$\dot e_i, \dot f_i, \dot
k_i^{-1}$).
We proceed by induction on $n$, setting $\psi(\dot v_i^+) = \dot
v_i^+$
and $\psi(\dot v_i^-) = \dot v_i^-$ in case $n=1$. 
For $n > 1$, let $\bar\bsigma$
and $\bar\bb$ denote the $(n-1)$-tuples
$(\sigma_1,\dots,\sigma_{n-1})$
and $(b_1,\dots,b_{n-1})$, respectively.
Assuming that the analog $\bar\psi$ of $\psi$ on the space
$\widehat{V}^{\otimes\bar\bsigma}$
has already been defined by induction,
we define $\psi$ on $\widehat{V}^{\otimes\bsigma}$ by setting
\begin{equation}\label{bardef}
\psi\left({\sum_{\bb \in \B}} p_\bb(q) \dot v_\bb\right) := 
{\sum_{\bb \in \B}} p_\bb(q^{-1})
\,\Theta\!\left(\bar\psi(\dot v_{\bar\bb}\right) \otimes \dot
v_{b_n}^{\sigma_n}).
\end{equation}

\begin{lemma}
The antilinear map $\psi$ defined by
(\ref{bardef}) is a well-defined involution of
$\widehat{V}^{\otimes\bsigma}$ preserving all weight spaces and
commuting with the actions of $\dot f_i, \dot e_i$ for all $i \in I$.
Moreover, $\psi(\dot v_\bb)$ is equal to $\dot v_\bb$ plus a (possibly infinite) 
$\Z[q,q^{-1}]$-linear combination of $\dot v_\ba$'s for $\ba \succ
\bb$.
\end{lemma}

\begin{proof}
Recall that $\Theta$ is a formal sum of terms
$\Theta_\beta$
for $\beta \in \bigoplus_{i \in I} \N \alpha_i$, with $\Theta_0 = 1$ and
$\Theta_\beta \in (U^-_q \mathfrak{sl}_\infty)_{-\beta} \otimes (U_q^+
\mathfrak{sl}_\infty)_\beta$.
The only monomials in the generators of $U_q^+\mathfrak{sl}_\infty$ that are non-zero
on $\dot v_{j}^{\sigma_n}$
are of the form $\dot e_i \dot e_{i+1}\cdots \dot e_{j-1}$
for $i \leq j$
if $\sigma_n=+$ (resp. the form $\dot e_{i-1} \dot e_{i-2} \cdots \dot e_{j}$
for $i \geq j$ if $\sigma_n=-$).
Using also the integrality of the quasi-$R$-matrix from 
\cite[Corollary 24.1.6]{Lubook} (or a direct calculation from \cite[Theorem
4.1.2(b)]{Lubook}),
it follows for any $v \in \widehat{V}^{\otimes\bar\bsigma}$ that 
\begin{equation}\label{rhs}
\Theta\left(v \otimes \dot
v_{j}^{\sigma_n}\right)
= v \otimes \dot
v^{\sigma_n}_{j}+
\sum_i \left(\Theta_{i,j} v\right) \otimes \dot v^{\sigma_n}_i
\end{equation}
summing over $i < j$ if $\sigma_n=+$ (resp. $i > j$ if
$\sigma_n=-$), 
for $\Theta_{i,j}$'s that are $\Z[q,q^{-1}]$-linear combinations of
monomials 
obtained by multiplying the generators
$\dot f_i, \dot f_{i+1},\dots, \dot f_{j-1}$
(resp. $\dot f_{i-1}, \dot f_{i-2}, \dots, \dot f_{j}$)
together in some order.
By induction, $\bar\psi(\dot v_{\bar\bb})$ equals $\dot v_{\bar\bb}$
plus a $\Z[q,q^{-1}]$-linear combination of $\dot v_{\bar\ba}$'s for
$\bar\ba \succ \bar\bb$.
Combining these two statements, we deduce that 
$$
\Theta\left(\bar\psi(\dot v_{\bar\bb}\right) \otimes \dot
v_{b_n}^{\sigma_n})
=
\dot v_{\bb}+\text{(a $\Z[q,q^{-1}]$-linear combination of
$\dot v_\ba$'s for $\ba \succ \bb$)}.
$$
This shows that the formula (\ref{bardef}) makes
$\psi(\dot v_\bb)$ into a well-defined element of
$\widehat{V}^{\otimes\bsigma}$ of the desired form.
The formula (\ref{bardef}) also makes sense for arbitrary sums $\sum_{\bb \in \B} p_\bb(q) \dot
v_\bb$
due to the interval-finiteness of the inverse dominance ordering
on $P^n$.
Finally, to see that $\psi$ commutes with the
actions of all $\dot f_i$ and $\dot e_i$, and that it is an
involution, one argues as in \cite[$\S$27.3.1]{Lubook}.
\end{proof}

Now we are in a position to apply ``Lusztig's Lemma''
as in the proof of \cite[Theorem 27.3.2]{Lubook}
to deduce 
for each $\bb \in \B$
that there is a unique vector
$\dot c_\bb \in \widehat{V}^{\otimes\bsigma}$
such that
\begin{itemize}
\item
$\psi(\dot c_\bb)=
\dot c_\bb$;
\item
$\dot c_\bb = \dot v_\bb + \text{a (possibly infinite)
$q\Z[q]$-linear combination of $\dot v_\ba$'s
for $\ba \succ \bb$}$.
\end{itemize}
This defines the {\em canonical basis} $\{\dot c_\bb\:|\:\bb \in \B\}$.
It is known (but non-trivial) that each $\dot c_\bb$
is always a {\em finite} sum of $\dot v_\ba$'s,
i.e. $\dot c_\bb \in \dot V^{\otimes \bsigma}$
before completion.
Moreover, the polynomials
$d_{\ba,\bb}(q)$ arising from the expansion
\begin{equation}
\dot c_\bb = {\sum_{\ba\in \B}} d_{\ba,\bb}(q) 
\dot v_\ba
\end{equation}
are some finite type A parabolic Kazhdan-Lusztig polynomials (suitably normalized).
All of these statements have a natural representation theoretic
explanation discussed in detail in \cite[$\S$5.9]{BLW}.
In particular, the results of \cite{BLW} (or \cite{CLW}) imply
the following.

\begin{theorem}\label{dnt}
Under the identification of $\mathbb{C} \otimes_{\mathbb{Z}} K_0(\O^\Delta)$ with $V^{\otimes \bsigma}$ from
Theorem~\ref{itacts}, 
$[P(\bb)]$ corresponds to the specialization $c_\bb$ of 
the canonical basis element $\dot c_\bb$ at $q=1$.
Equivalently,
$(P(\bb):M(\ba)) = [M(\ba):L(\bb)] = d_{\ba,\bb}(1)$ for each $\ba,\bb \in \B$.
\end{theorem}

We are ready to explain our new algorithm to compute $\dot c_\bb$.
We proceed by induction on $n$.
In case $n=1$, we have that $\dot c_\bb = \dot v_\bb$ always.
If $n > 1$, we first 
compute $\dot c_{\bar\bb} \in \widehat{V}^{\otimes\bar\bsigma}$.
It is a linear combination of
finitely many $\dot v_{\bar\ba}$'s for $\bar\ba \succeq \bar\bb$.
Then we define $j \in I$ as follows.
\begin{itemize}
\item
If $\sigma_n=+$ then $j$ is the greatest integer such that $j \leq b_n$, and the following hold for all
$1 \leq r < n$ and all tuples
$\bar\ba = (a_1,\dots,a_{n-1})$ such that $\dot v_{\bar\ba}$ occurs in the expansion of $\dot c_{\bar\bb}$:
\begin{itemize}
\item[$\circ$] if $\sigma_r=+$ then $j \leq a_r$;
\item[$\circ$] if $\sigma_r=-$ then $j < a_r$.
\end{itemize}
\item
If $\sigma_n=-$ then $j$ is the smallest integer such that $j \geq b_n$, and the following hold for all
$1 \leq r < n$ and all tuples
$\bar\ba = (a_1,\dots,a_{n-1})$ such that $\dot v_{\bar\ba}$ occurs from the expansion of $\dot c_{\bar\bb}$:
\begin{itemize}
\item[$\circ$] if $\sigma_r=-$ then $j \geq a_r$;
\item[$\circ$] if $\sigma_r=+$ then $j > a_r$.
\end{itemize}
\end{itemize}

\begin{lemma}\label{lastofyear}
In the above notation, we have that
$\Theta\left(\dot c_{\bar \bb} \otimes \dot
v_j^{\sigma_n}\right) = \dot c_{\bar\bb} \otimes \dot v_j^{\sigma_n}$.
\end{lemma}

\begin{proof}
By (\ref{rhs}), we have that
$\Theta\left(\dot c_{\bar \bb} \otimes \dot
v_j^{\sigma_n}\right) = \dot c_{\bar\bb} \otimes \dot v_j^{\sigma_n}
+ \sum_i \left(\Theta_{i,j} \dot c_{\bar\bb} \right) \otimes \dot v_i^{\sigma_n}$
summing over $i < j$ if $\sigma_n=+$ (resp. $i > j$ if $\sigma_n=-$),
where
$\Theta_{i,j}$ is a linear combination of non-trivial monomials
in the generators $\dot f_{j-1}, \dot f_{j-2},\dots, \dot f_i$ (resp. $\dot f_j, \dot
f_{j+1},\dots, \dot f_{i-1}$).
By the definition of $j$, all of these generators act as zero on $\dot c_{\bar\bb}$.
\end{proof}

Lemma~\ref{lastofyear} shows
that the vector 
$\dot c_{\bar \bb} \otimes \dot v_j^{\sigma_n}\in
\widehat{V}^{\otimes\bsigma}$ is fixed by $\psi$.
Hence, so too is 
$X\left(\dot c_{\bar \bb} \otimes \dot v_j^{\sigma_n}\right)$
where
$$
X := 
\left\{
\begin{array}{ll}
\dot f_{b_n-1} \cdots \dot f_{j+1} \dot f_{j}&\text{if $\sigma_n = +$,}\\
\dot f_{b_n} \cdots \dot f_{j-2} \dot f_{j-1}&\text{if $\sigma_n = -$.}
\end{array}\right.
$$
By Lemma~\ref{not}, this new vector equals $\dot v_\bb$ plus a $\Z[q,q^{-1}]$-linear combination of $\dot
v_{\ba}$'s for $\ba \succ \bb$. If all but its leading coefficient lie in $q \Z[q]$, it is
already the desired vector $\dot c_{\bb}$.
Otherwise, one picks $\ba
\succ \bb$ minimal so that the $\dot v_\ba$-coefficient is not in $q
\Z[q]$, then subtracts a bar-invariant multiple of the recursively
computed vector $\dot c_\ba$ to remedy this defficiency. Continuing in
this way, we finally obtain a bar-invariant vector with all of the required
properties to be
$\dot c_{\bb}$.

The algorithm just described has been implemented in
{\sc Gap}, and is available at
{\tt http://pages.uoregon.edu/brundan/papers/A.gap}.
It is not obvious to us that it terminates in
finite time for every $\bb \in \B$. Based on examples, we believe that this is indeed the
case.

\begin{example}
Suppose $\bsigma = (+,+,-,-)$ and $\bb = (1,2,2,1)$.
By induction, we have $\dot c_{(1,2,2)} = 
\dot v_{(1,2,2)} + q \dot
v_{(2,1,2)} + q \dot v_{(1,3,3)} + q^2 \dot v_{(3,1,3)}$, so take $j=4$.
We compute
\begin{multline*}
\dot f_1 \dot f_2 \dot f_3 
(\dot v_{(1,2,2,4)} + q \dot
v_{(2,1,2,4)} + q \dot v_{(1,3,3,4)} + q^2 \dot v_{(3,1,3,4)})=
\dot v_{(1, 2, 2, 1)} 
+q\dot v_{(1, 2, 1, 2)}+
q\dot v_{(2, 1, 2, 1)}
\\+q\dot v_{(1, 4, 1, 4)}
+ q\dot v_{(1, 3, 3, 1)}
+q^2\dot v_{(2, 1, 1, 2)}
+q^2\dot v_{(3, 1, 3, 1)}+
q^2 \dot v_{(2, 3, 3, 2)}
+q^2 \dot v_{(4, 1, 1, 4)}
\\\:+
q^2\dot v_{(2, 4, 2, 4)}
+ (1+q^2) \dot v_{(1, 3, 1, 3)}
+ q^3\dot v_{(3, 2, 3, 2)}
+q^3\dot v_{(4, 2, 2, 4)}
+(q+q^3)\dot v_{(2, 3, 2, 3 )}\\+
(q+q^3) \dot v_{(3, 1, 1, 3)}
+ (q+q^3) 
\dot v_{(2, 2, 2, 2)}
+(q^2+q^4) \dot v_{(3, 2, 2, 3)}.
\end{multline*}
Then we subtract the recursively computed
\begin{multline*}
\dot c_{(1,3,1,3)}=
\dot v_{(1, 3, 1, 3)}
+
q \dot v_{(3, 1, 1, 3)}
+
q
\dot v_{(2, 3, 2, 3)}
+
q
\dot v_{(1, 4, 1, 4)}\\
+q^2 \dot v_{(3, 2, 2, 3)}
+q^2 \dot v_{(4, 1, 1, 4)} +  q^2 \dot v_{(2, 4, 2, 4)}
+
q^3 \dot v_{(4, 2, 2, 4)}
\end{multline*}
to get $\dot c_{(1,2,2,1)}$.
\end{example}


\begin{thebibliography}{CKW}
\bibitem[B1]{B0}
J. Brundan,
Kazhdan-Lusztig polynomials and character formulae for the Lie
superalgebra $\mathfrak{gl}(m|n)$,
 {\em J. Amer. Math. Soc.} {\bf 16} (2003),
185--231.

\bibitem[B2]{B1}
J. Brundan,
Kazhdan-Lusztig polynomials and character formulae for the Lie superalgebra $\q(n)$,
{\em Advances Math.} {\bf 182} (2004), 28--77.

\bibitem[B3]{B2}
J. Brundan,
Tilting modules for Lie superalgebras,
{\em Commun. Alg.} {\bf 32} (2004), 2251--2268.

\bibitem[BD]{BDC}
J. Brundan and N. Davidson,
Type C blocks of super category $\O$, in preparation.


\bibitem[BE1]{BE}
J. Brundan and A. Ellis,
Monoidal supercategories;
\arxiv{1603.05928}.

\bibitem[BE2]{BE2}
J. Brundan and A. Ellis,
Super Kac-Moody 2-categories, in preparation.

\bibitem[BK]{BK}
J. Brundan and A. Kleshchev, Hecke-Clifford superalgebras,
crystals of type $A_{2\ell}^{(2)}$ and modular branching rules for 
$\widehat{S}_n$,
{\em Represent. Theory} {\bf 5} (2001), 317-403.

\bibitem[BLW]{BLW}
J. Brundan, I. Losev, and B. Webster,
Tensor product categorifications and the super Kazhdan-Lusztig
conjecture,
to appear in {\em IMRN}.

\bibitem[BGS]{BGS}
A. Beilinson, V. Ginzburg and W. Soergel,
Koszul duality patterns in representation theory,
{\em J. Amer. Math. Soc.} {\bf 9} (1996), 473--527.

\bibitem[CL]{CL}
B. Cao and N. Lam,
An inversion formula for some Fock spaces; \arxiv{1512.00577v2}.

\bibitem[C]{C}
C.-W. Chen, 
Reduction method for representations of queer Lie superalgebras,
{\em J. Math. Phys.} {\bf 57} 051703 (2016).

\bibitem[CC]{CC}
C.-W. Chen and S.-J. Cheng,
    Quantum group of type $A$ and representations of queer Lie
    superalgebra;
\arxiv{1602.04311}.

\bibitem[CK]{CK}
S.-J. Cheng and J.-H. Kwon,
Finite-dimensional half-integer weight modules over queer Lie
superalgebras,
{\em Commun. Math. Phys.}
{\bf 346} (2016), 945--965.

\bibitem[CKW]{CKW}
S.-J. Cheng, J.-H. Kwon, and W. Wang, 
Character formulae for queer Lie superalgebras and canonical bases of
types A/C,
to appear in {\em Commun. Math. Phys.}.

\bibitem[CLW]{CLW}
S.-J. Cheng, N. Lam and W. Wang,
Brundan-Kazhdan-Lusztig conjecture for general linear Lie
superalgebras, 
{\em Duke Math. J.} {\bf 164} (2015), 617--695.

\bibitem[CW]{CW}
S.-J. Cheng and W. Wang, 
\textit{Dualities and Representations of Lie Superalgebras,}
 Graduate Studies in Mathematics 144, 
AMS, 2012.   

\bibitem[CR]{CR} 
J. Chuang and R. Rouquier,
Derived equivalences for symmetric groups and $\mathfrak{sl}_2$-categorification, {\em Ann. Math.} {\bf 167} (2008), 245--298.

\bibitem[CPS]{CPS}
E. Cline, B. Parshall and L. Scott,
Finite dimensional algebras
and highest weight categories,
{\em J. Reine Angew. Math.} {\bf 391} (1988), 85--99.

\bibitem[F]{F}
A. Frisk, 
Typical blocks of the category $\O$ for the queer Lie superalgebra,
{\em J. Algebra Appl.} {\bf 6} (2007), 731--778.

\bibitem[H]{H}
J. Humphreys, {\em Representations of Semisimple Lie Algebras in the BGG Category $\mathcal O$}, Graduate Studies in Mathematics 94, AMS, 2008.
 

\bibitem[KKT]{KKT}
S.-J. Kang, M. Kashiwara and  S. Tsuchioka,
Quiver Hecke superalgebras,
{\em J. Reine Angew. Math.} {\bf 711} (2016), 1--54. 

\bibitem[KL]{KL1}
M. Khovanov and A. Lauda,
A diagrammatic approach to categorification of quantum
groups I, {\em Represent. Theory} {\bf 13} (2009), 309--347.

\bibitem[HKS]{HKS}
D. Hill, J. Kujawa and J. Sussan,
Degenerate affine Hecke-Clifford algebras and type Q Lie superalgebras,
{\em Math. Zeit.} {\bf 268} (2011), 1091--1158.

\bibitem[LW]{LW}
I. Losev and B. Webster,
On uniqueness of tensor products of irreducible categorifications, 
{\em Selecta Math.} {\bf 21} (2015), 345--377. 

\bibitem[L]{Lubook}
G. Lusztig, {\em Introduction to Quantum Groups}, Birkh\"auser, 1993.

\bibitem[M]{M}
V. Mazorchuk,
Parabolic category $\mathcal O$ for classical Lie superalgebras,
in: {\em Advances in Lie Superalgebras}, M. Gorelik and P. Papi
(eds.), Springer INdAM Series 7, Springer, 2014, pp. 149--166.

\bibitem[N]{N}
M. Nazarov, 
Young's symmetrizers for projective representations of the symmetric group,
{\em Advances Math.} {\bf 127} (1997), 190--257. 

\bibitem[R]{Rou}
R. Rouquier, 
2-Kac-Moody algebras;
\arxiv{0812.5023}.

\bibitem[W]{Web}
B. Webster,
Knot invariants and higher representation theory, to appear in {\em
  Mem. Amer. Math. Soc.}.
\end{thebibliography}
\end{document}